\documentclass[12pt,oneside,a4paper]{amsart}

\usepackage{fullpage}
\usepackage[left=2.2cm,right=2.2cm,top=2.2cm,bottom=2.2cm,footskip=1cm,includeheadfoot]{geometry}

\usepackage{hyperref}
\usepackage{amssymb,amsmath,mathrsfs,amsthm}
\usepackage[dvipsnames]{xcolor}
\usepackage{etoolbox}
\usepackage{fullpage}
\usepackage{calrsfs}
\usepackage[T1]{fontenc} 
\usepackage{amsfonts}
\usepackage[shortlabels]{enumitem}
\usepackage{mathtools}
\usepackage{bm}
\usepackage{tikz-cd}
\setenumerate{label={\normalfont(\textit{\roman*})}, ref={\textrm{\roman*}}
}
\patchcmd{\section}{\scshape}{\bfseries}{}{}
\makeatletter
\renewcommand{\@secnumfont}{\bfseries}
\makeatother

\DeclareRobustCommand{\SkipTocEntry}[5]{}

\newtheorem{introtheorem}{Theorem}

\theoremstyle{definition}

\newtheorem*{introdefinition*}{Definition}
\theoremstyle{plain}

\newcommand\Z{{\mathbf Z}}
\newcommand\N{{\mathbf N}}

\theoremstyle{plain}
\newtheorem{theorem}{Theorem}[section]
\newtheorem{proposition}[theorem]{Proposition}
\newtheorem{lemma}[theorem]{Lemma}
\newtheorem{corollary}[theorem]{Corollary}

\newtheorem{conjecture}[theorem]{Conjecture}
\theoremstyle{definition}
\newtheorem{definition}[theorem]{Definition}
\newtheorem{example}[theorem]{Example}
\newtheorem{remark}[theorem]{Remark}

\renewcommand{\geq}{\geqslant}
\renewcommand{\leq}{\leqslant}
\newcommand{\A}{\mathcal{A}}
\newcommand{\B}{\mathcal{B}}
\newcommand{\F}{\mathrm{F}}
\newcommand{\R}{\mathrm{R}}
\newcommand{\abs}[1]{\left| #1 \right|}

\AtBeginDocument{%
   \def\MR#1{}
}

\begin{document}

\date{\today}
 \title{Quasi-fixed points of substitutive systems}

\author[El\.zbieta Krawczyk]{El\.zbieta Krawczyk}

\address[EK]{Faculty of Mathematics\\
University of Vienna\\
Oskar Morgensternplatz 1\\
1090 Vienna \& 
Faculty of Mathematics and Computer Science\\Institute of Mathematics\\
Jagiellonian University\\
Stanis\l{}awa \L{}ojasiewicza 6\\
30-348 Krak\'{o}w}
\email{ela.krawczyk7@gmail.com}

\subjclass[2010]{Primary: 11B85, 37B10, 68R15. Secondary:  68Q45}
\keywords{automatic sequence, substitutive sequence, quasi-fixed point, factor map}

\maketitle
\begin{abstract}
	We study automatic sequences and automatic systems generated by general constant length (nonprimitive) substitutions.
	 While an automatic system is typically uncountable, the set of automatic sequences is countable, implying that most sequences within an automatic system are not themselves automatic. 
	We provide a complete and succinct classification of automatic sequences that lie in a given automatic system in terms of the quasi-fixed points of the substitution defining the system. 
	Our result extends to factor maps between automatic systems and highlights arithmetic properties underpinning these systems.
	We conjecture that a similar statement holds for general nonconstant length substitutions.
 \end{abstract}

\section*{Introduction}
	Automatic sequences are ubiquitous in mathematics, both as a general object of study and as a useful tool; their particularly famous examples include the Thue--Morse, the Rudin--Shapiro, or the Baum--Sweet sequence, see e.g.\ \cite[Sec.\ 5.1]{AlloucheShallit-book}.
	Since their introduction in 1969  by Cobham \cite{Cobham-1969}, they have proved to be an important concept in many algebraic, combinatorial, computational, and number-theoretic contexts \cite{AlloucheShallit-book, Fogg, ABsurvey}.

	Instead of a single automatic sequence $x\in\A^{\Z}$, one can consider its orbit closure under the shift map $T$, that is
	 \begin{align*}
	 \overline{\mathcal{O}(x)} & =\overline{\{T^n(x)  \mid n\in\Z\}} \\
	 & = \{z\in \A^{\Z}\mid \text{ all finite words appearing in } z \text{ appear also in } x\}.
	 \end{align*}
	This is a classical object studied in symbolic dynamics, which---via the second line above---also has a clear combinatorial interpretation.
	It is also closely related to DOL-systems studied in computer science \cite{RS-book}.	  
	Conversely, many combinatorial properties of the sequence itself can be studied via the symbolic system generated by it.
	While such systems are classically assumed to be minimal (corresponding to the substitution defining the sequence $x$ being primitive) \cite{Lothaire-book1,Lothaire-book2,Queffelec-book}, recent years show growing interest in the study of nonminimal substitutive systems \cite{MR,BKM, BSTY, BPRS2023, BKK, BPR24}.
	Furthermore, from a combinatorial or number-theoretic point of view minimality is not a natural assumption: many sequences naturally appearing in this contexts---such as e.g.\ the Baum--Sweet sequence or the family of sparse automatic sequences \cite{thesisSeda}---do not give rise to minimal systems. 
	
	 The problem we address in this paper is the following: which sequences in $\overline{\mathcal{O}(x)}$ are again automatic? 
	In Theorem \ref{introthm:characteraytomaticsequences} we obtain a complete answer to this question in terms of the constant length substitution $\varphi$ defining the sequence $x$.

 	Formally, a \textit{substitution} over an \emph{alphabet} $\mathcal{A}$ is a map $\varphi\colon \mathcal{A}\rightarrow \mathcal{A}^*$ that assigns to each letter $a\in\mathcal{A}$ some finite word $\varphi(a)$ in $\mathcal{A}^*$.\footnote{We always assume that $\varphi$ is \emph{growing}, i.e.\ the lengths of the words $\varphi^n(a)$ tend to infinity as $n\to\infty$ for all $a\in\A$.}
 	A substitution $\varphi$ is \emph{primitive} if there is $n\geq 1$ such that all letters from $\A$ appear in $\varphi^n(b)$ for each $b\in\A$.
 	A substitution  $\varphi$ is \textit{of (constant) length} $k$ if the length of $\varphi(a)$ is $k$ for each $a\in \mathcal{A}$.  A \textit{coding} is an arbitrary map $\tau\colon \mathcal{A}\rightarrow \mathcal{B}$ between two alphabets $\mathcal{A}$ and $\mathcal{B}$.  One can extend a substitution or a coding to a map acting on biinfinite sequences by concatenation in a natural way.
 	A point $x\in \A^{\Z}$ is called $\varphi$-\emph{periodic} if $\varphi^n(x)=x$ for some $n\geq 1$; if one can take $n=1$, then $x$ is said to be a \emph{fixed point of} $\varphi$.
 	A  sequence in $\mathcal{A}^{\mathbf{Z}}$ is called  \textit{substitutive} if it is the image under a coding of a fixed point of some substitution $\varphi\colon\mathcal{A}\rightarrow\mathcal{A}^*$. A substitutive sequence is $k$-\emph{automatic} if we may take $\varphi$ to be of constant length $k$.

	One may study substitutions by the symbolic systems they generate:  for a substitution $\varphi\colon \A\to\A^{*}$ we let 
\[X_{\varphi}=\{x\in\mathcal{A}^{\mathbf{Z}}\mid \textrm{ all factors of } x \textrm{ appear in } \varphi^n(a) \textrm{ for some } a\in\mathcal{A},\ n\geq 1\}\]
be the system generated by  $\varphi$. 
	A system $X$ is \emph{substitutive} if $X=\tau(X_{\varphi})$ for some substitution $\varphi\colon\A\to\A^{*}$ and coding $\tau\colon\A\to\B$; it is $k$-\emph{automatic} if one can take $\varphi$ to be of constant length $k$.
	 Note that for a substitutive (resp.\ automatic) sequence $x$ its orbit closure $\overline{\mathcal{O}(x)}$
	 is clearly a substitutive (resp.\ automatic) system.
	One of the central question is that of classification: when are two substitutive systems isomorphic as dynamical systems?; when is one system a (topological) factor of the other?; what is the automorphism group of a substitutive system?\footnote{For the definitions of morphism notions in  the category of topological dynamical systems see Section \ref{sec:dynamics}.} 
	Despite recent interest and some satisfying answers \cite{CDK17, CDK14, CKL08, Coven2015, ST-2015, DL-2018, MY-21}, a lot of interesting questions still remain open \cite[Sec.\ 10]{DL-2018}.

	Let $X$ be a substitutive (resp.\ automatic) system generated by some sequence $x$.
	It is not hard to see that one can choose the generating sequence $x$ to be substitutive (resp.\ automatic) itself (see Section \ref{sec:substsystems}).	
	However, such a sequence is never \emph{unique} e.g.\ every shift of $x$ is also a substitutive (resp.\ automatic) point generating $X$. 
	Since most dynamical systems (e.g.\ all nonperiodic recurrent systems \cite[Thm.\ 10.8.12]{AlloucheShallit-book}) have uncountably many points and there are only countably many substitutive sequences, most sequences in a substitutive system are not substitutive.
	 Furthermore, substitutive (resp.\ automatic) sequences are preserved under factor maps between subshifts and as such provide strong constraints for an existence of a factor map between two substitutive systems. 
	In view of this, a natural problem is to characterise the set of substitutive sequences that lie in a given substitutive system.
	 
	For a substitution $\varphi\colon\A\to\A^*$, we say that a point $x\in\A^{\Z}$ is a \emph{quasi-fixed point} of $\varphi$ if 
	\[x=T^c\varphi^m(x)\quad \text{for some } m\geq 1 \text{ and } c\in\Z.\]
	Quasi-fixed points have appeared implicitly in a few places  and  already proved to be a useful tool in the study of substitutive systems. 
	They appeared e.g.\ \begin{itemize}
\item in \cite{CRSZ, ARS09} in the study of substitutivity of lexicographically minimal/maximal points  in  substitutive systems; 
\item in \cite{HZ-01} in the study of certain directed graphs associated with primitive substitutions; 
\item in \cite{AP2024} in the study of the dimensions of the lower central series factors of a certain Lie algebra; 
\item in \cite{BKK} in the study of subsystems of substitutive systems and the finitary Cobham's theorem;
\item or in \cite{Coven2015} in the study of automorphism group of minimal automatic systems (disguised there as rational points in $\Z_k$, see Remark  \ref{rem:conjugacyproblemsviaquasifixed}).
\end{itemize}
	 They also appeared explicitly in  \cite{twosidedfixedpointsSW} (see also \cite[Sec.\ 7.4]{AlloucheShallit-book}) or more recently in \cite[Sec.\ 5]{BPR24}, where a detailed characterisation of such points was given.

	 Our main result---Theorem \ref{introthm:characteraytomaticsequences} below--- shows that the set of automatic  sequences in an automatic system $X_{\varphi}$ admits a very simple description in terms of the substitution $\varphi$: it comprises precisely of the quasi-fixed points of $\varphi$.
	 In fact, we work in a somewhat more general framework and obtain that all  lifts of nonperiodic automatic sequences by factor maps between automatic  systems are automatic.  
	 In particular, if $\pi\colon X_{\varphi}\to X_{\varphi'}$ is a factor map between two automatic  systems, then $\pi$ maps quasi-fixed points of $\varphi$ to quasi-fixed points of $\varphi'$, and all  preimages of nonperiodic quasi-fixed points of $\varphi'$ by $\pi$ are quasi-fixed points of $\varphi$.
	 In the case when $\varphi$ is primitive, its set of automatic sequences can be characterised in terms of the arithmetic properties of the ring $\Z_k$ of $k$-adic integers which underpins the dynamics of $X_{\varphi}$.
	 We note here that, by Fagnot's generalisation of Cobham's theorem \cite[Thm.\ 15]{Fagnot-1997}, all automatic sequences that appear in a $k$-automatic system are necessarily $k$-automatic.
	
\begin{introtheorem}[Thm.\ \ref{thm:characteraytomaticsequences} \& Cor.\ \ref{cor:automaticcorrespondtorational}] \label{introthm:characteraytomaticsequences}
Let $k\geq 2$. Let $\varphi$ be a substitution of length $k$. Let $\pi\colon X_{\varphi}\to Y$ be a factor map onto some subshift $Y$. The following hold.\begin{enumerate}
\item  A sequence $y\in Y$ is $k$-automatic if and only if $y=\pi(x)$ for some quasi-fixed point $x$ of $\varphi$.
\item  If $y\in Y$ is $k$-automatic and nonperiodic, then all  points in $\pi^{-1}(y)$ are quasi-fixed points of $\varphi$.
\end{enumerate}
Furthermore, if the substitution $\varphi$ is primitive and $X_{\varphi}$ is infinite, then $k$-automatic sequences in $X_{\varphi}$ correspond precisely to the rationals in $\Z_k$ under the unique factor map $X_{\varphi}\to\Z_k$ which sends $\varphi$-periodic points to $0\in\Z_k$.\footnote{See Section \ref{sec:k-adic} for the definition of the factor map.} 
\end{introtheorem}

	Letting the factor map $\pi$ be a coding in Theorem \ref{introthm:characteraytomaticsequences}  treats the case of automatic systems; letting furthermore the coding $\pi$ be the identity treats the case of purely automatic systems $X_{\varphi}$. 
	The same statement is true if one works with one-sided systems instead of two-sided; the one-sided case is treated in the appendix at the end of the paper.

	We believe that the first part of Theorem \ref{introthm:characteraytomaticsequences} holds for general (nonconstant length) substitutive systems (see Conjecture \ref{con:characteraytomaticsequences}). 
	Indeed, it is true in the case of minimal substitutive systems; it can be deduced from \cite{HZ-01} and Section \ref{sec:factors} (see Remark \ref{rem:holton}).
	The last statement in Theorem  \ref{introthm:characteraytomaticsequences} can  be deduced from \cite{Coven2015} (see Remark \ref{rem:conjugacyproblemsviaquasifixed}).
	Both approaches rely heavily on the minimality of the system. 
	The proof in \cite{HZ-01} uses the characterisation of primitive substitutive sequences in terms of derived sequences and return words due to Durand \cite{D-derived}.  However, these methods do not  to work well in the nonminimal case when the sets of return words are infinite.
	In \cite{Coven2015} the existence of $\Z_k$ as an equicontinous factor of a minimal nonperiodic $k$-automatic system is crucial, a fact which no longer holds if the system is not minimal (see Example \ref{ex:automaticsystemnotsubstitutive}).

\section*{Acknowledgements} We thank Reem Yassawi for her comments and bringing some references to our attention as well as Maik Gr\"oger for his comments and support.
	 This research was funded in whole or in part by National Science Center, Poland under grant no.\ UMO2021/41/N/ST1/04295. For the purpose of Open Access, the author has applied a CC-BY public copyright licence to any
Author Accepted Manuscript (AAM) version arising from this submission.

\section{Substitutive and automatic systems}\label{sec:preliminaries}
	
	In this section we present some preliminary definitions and lemmas which will be used throughout the rest of the paper. 	
	Since there is no systematic study of general nonminimal substitutive systems yet, we give some proofs extending the facts known in the minimal case to the general set-up considered by us.

\subsection{Words and substitutions}\label{sec:words}  Let $\mathcal{A}$ be an \emph{alphabet}, that is a finite set of symbols.
	 We denote by $\mathcal{A}^*$ the set of finite words over $\mathcal{A}$. This is a monoid under concatenation; the empty word is denoted by $\epsilon$.
	  We denote by $\mathcal{A}^{\mathbf{N}}$ the set of infinite sequences over $\mathcal{A}$, where $\mathbf{N}=\{0,1,2,\ldots\}$ stands for the set of nonnegative integers, and by $\mathcal{A}^{\mathbf{Z}}$ the set of biinfinite sequences. 
	  A one-sided or two-sided sequence $x=(x_n)_n$ is called \emph{periodic} if $x_{n}=x_{n+p}$ for some $p\geq 1$ and all $n$. 
	  A one-sided sequence $x=(x_n)_{n\in\N}$ is called \emph{ultimately periodic} if the sequence $(x_{m+n})_{n\in\N}$ is periodic for some $m\geq 0$.
	   For a word $u\in\mathcal{A}^*$ we denote by $|u|$ the length of $u$. All finite words  are indexed starting at $0$. We say that a finite word $w$ is a \emph{factor} of a sequence or a finite word $x$ if $w$ appears somewhere in $x$, that is, 
\[w=x_ix_{i+1}\dots x_{i+t-1},
\]
for some $i$, where $t=\abs{w}$. 
	We let $\mathcal{L}(x)$ (resp.\ $\mathcal{L}^r(x)$) denote the the \emph{language} (resp.\ $r$-\emph{language}) of $x$ consisting of all finite words (resp.\ all words of length $r$) appearing in $x$.

Let $\mathcal{A}$ and $\mathcal{B}$ be alphabets. A \textit{morphism} is a map $\varphi\colon \mathcal{A}\rightarrow \mathcal{B}^*$ that assigns to each letter $a\in\mathcal{A}$ some finite word $w$ in $\mathcal{B}^*$. A morphism  $\varphi$ is \textit{of (constant) length} $k$ if $|\varphi(a)|=k$ for each $a\in \mathcal{A}$.  A \textit{coding} is a morphism of constant length 1, i.e.\ an arbitrary map $\tau\colon \mathcal{A}\rightarrow \mathcal{B}$. If $\mathcal{A}=\mathcal{B}$, we refer to any morphism  $\varphi$ as \textit{substitution}. We will always assume that substitutions are \textit{growing}, i.e.\ $|\varphi^n(a)|$ tends to infinity as $n\to\infty$ for all $a\in\mathcal{A}$.

 A letter $a\in \mathcal{A}$ is \textit{right-prolongable} (w.r.t.\ the substitution $\varphi$) if $a$ is the initial letter of $\varphi(a)$. If $a$ is right-prolongable, then the sequence $\varphi^n(a)$ converges to a sequence in $\mathcal{A}^{\mathbf{N}}$  which we denote by $\varphi^{\omega}(a)$.
	Similarly, a letter $a\in\mathcal{A}$ is \textit{left-prolongable} (w.r.t. the substitution $\varphi$) if $a$ is the final letter of $\varphi(a)$; $\varphi^n(a)$ converges then to a left-infinite  sequence which we denote by  ${}^{\omega}\!\varphi(a)$ (see also Definition \ref{def:quasifixedpointexplicitformula} later).

A morphism $\varphi\colon \mathcal{A}\rightarrow \mathcal{B}^*$ induces natural maps $\varphi\colon \mathcal{A}^{\mathbf{N}}\to \mathcal{B}^{\mathbf{N}}$ and $\varphi\colon \mathcal{A}^{\mathbf{Z}}\to \mathcal{B}^{\mathbf{Z}}$; in the latter case, the map is given by the formula
\[\varphi(\ldots z_{-1}.z_{0}\ldots)=\ldots \varphi(z_{-1}).\varphi(z_0)\ldots,\] where the dot indicates the 0th position.  For a substitution $\varphi\colon \mathcal{A}\rightarrow \mathcal{A}^*$, a sequence $z$ in $\mathcal{A}^{\mathbf{N}}$ or in $\mathcal{A}^{\mathbf{Z}}$ is called a $\varphi$-\textit{periodic point} if $\varphi^n(z)=z$ for some $n\geq 1$. If one can take $n=1$, then $z$ is a \textit{fixed point of} $\varphi$. Since $\varphi$ is assumed to be growing, all one-sided fixed points of $\varphi$ are given by $\varphi^{\omega}(a)$ for some right-prolongable letter $a$, and all two-sided fixed points of $\varphi$ are given by ${}^{\omega}\!\varphi(b).\varphi^{\omega}(a)$ for some  right-prolongable letter $a$ and left-prolongable letter $b$.

Let $X$ be a set. Recall that a map $f\colon X\to X$ is \textit{idempotent} if $f^2=f$. If $X$ is a finite set, then for any map $f\colon X\to X$, its $r$th iterate $f^r$ is idempotent, where $r=|X|!$ (see e.g.\ \cite[Lemma 1.7]{BKK}). 

\begin{definition}\label{def:ambiidempotent}
We say that a substitution $\varphi\colon \mathcal{A}\rightarrow \mathcal{A}^*$ is \textit{ambi-idempotent} if for all $a\in \mathcal{A}$, the first letter of $\varphi(a)$ is right-prolongable and the last letter of $\varphi(a)$ is left-prolongable, i.e.\ the maps \[F_{\varphi}\colon\mathcal{A}\to\mathcal{A},\ a\mapsto\varphi(a)_0 \quad \text{and}\quad G_{\varphi}\colon \mathcal{A}\to \mathcal{A},\ a\mapsto \varphi(a)_{|\varphi(a)|-1}\] are idempotent.
\end{definition}

For an ambi-idempotent substitution  $\varphi\colon \mathcal{A}\rightarrow \mathcal{A}^*$ and  $n\geq 1$, $a\in\mathcal{A}$ is right-prolongable (resp.\ left-prolongable)  with respect to $\varphi^n$ if and only if $a$ is  right-prolongable (resp.\ left-prolongable) with respect to $\varphi$. Hence, for an ambi-idempotent substitution  $\varphi$, all  $\varphi$-periodic points are fixed points of $\varphi$.
	Any substitution has some power that is ambi-idempotent.

\begin{lemma}\label{lem:idempotent} Let  $\varphi\colon \mathcal{A}\rightarrow \mathcal{A}^*$ be a substitution. Let $r=|\mathcal{A}|!$. The substitution $\varphi^r$ is ambi-idempotent.
\end{lemma}
\begin{proof}  It is enough to note that $F^n_{\varphi}=F_{\varphi^{n}}$ and $G^n_{\varphi}=G_{\varphi^{n}}$ for all $n\geq 1$, and that both functions $F_{\varphi}^r$ and $G_{\varphi}^r$ are idempotent, e.g.\ by \cite[Lemma 1.7]{BKK}.
\end{proof}

\subsection{Substitutive sequences}\label{sec:subsseq}

Let $k\geq 2$. A (one-sided or two-sided) fixed point of a substitution (resp.\  substitution of constant length $k$) is called a \textit{purely substitutive} (resp.\  \textit{purely}  $k$-\textit{automatic}) sequence.
	 A sequence is  \textit{substitutive} (resp.\ $k$-\textit{automatic}) if it can be obtained as the image of a purely substitutive (resp.\ purely $k$-automatic) sequence under a coding. We will also say that a left-infinite sequence $(x_n)_{n<0}$ is substitutive (resp.\ $k$-automatic) if the  right-infinite sequence $(x_{-n-1})_{n\geq 0}$ is substitutive (resp.\ $k$-automatic).

 We will frequently use the following closure properties of substitutive sequences; unless otherwise specified the claims hold for both one-sided and two-sides sequences.

\begin{lemma}\label{lem:closure}
\begin{enumerate}
\item\label{lem:closure1}\cite[Lem.\ 2.10]{BKK} A two-sided sequence $x$ is substitutive (resp.\ $k$-automatic) if and only if  the one-sided sequences $(x_n)_{n\geq 0}$ and $(x_n)_{n<0}$ are substitutive (resp.\ $k$-automatic).
\item\label{lem:closure2} \cite[Thm.\ 7.6.1, Thm.\ 7.6.3 and Cor.\ 6.8.5]{AlloucheShallit-book}  The sets of substitutive (resp.\ $k$-automatic) sequences are closed under the left and right shifts.
\item\label{lem:closure4} \cite[Thm.\ 5.4.2]{AlloucheShallit-book} All ultimately periodic one-sided sequences are $k$-automatic with respect to any $k\geq 2$, and thus all two-sided periodic sequences are $k$-automatic for any $k\geq 2$.
\item\label{lem:closure3} \cite[Thm.\ 6.6.4]{AlloucheShallit-book}  For each $n\geq 1$, a sequence $x$ is $k$-automatic if and only if it is $k^n$-automatic. 
\item\label{lem:closure5} \cite{Cobham72}, \cite[Thm.\ 6.6.2]{AlloucheShallit-book} A sequence $x=(x_n)_n$ is $k$-automatic if and only if its $k$-\emph{kernel}
\[\mathrm{K}_k(x)=\{ (x_{k^{m}n+i})_{n} \mid m\geq 0, 0\leq i\leq k^m-1\}\]
is finite. 
\end{enumerate}
\end{lemma} 
\begin{proof}
To see the first claim assume that $x$ is substitutive. Then $x=\tau({}^{\omega}\!\varphi(b)).\tau(\varphi^{\omega}(a))$ for some substitution $\varphi\colon\mathcal{A}\to\mathcal{A}^*$, coding $\tau\colon\mathcal{A}\to\mathcal{B}$, right-prolongable letter $a$, and left-prolongable letter $b$; furthermore if $x$ is $k$-automatic, then $\varphi$ can be taken to be of constant length $k$. Hence, the one-sided sequences  $(x_n)_{n\geq 0}$ and $(x_n)_{n<0}$ are substitutive, and they are $k$-automatic if $x$ is $k$-automatic. The other implication follows from (the proof of) \cite[Lemma 2.10]{BKK}.\footnote{Note that in \cite{BKK} a two-sided substitutive sequence is \emph{defined as} a two-sided sequence $x$ such that the one-sided sequences $(x_n)_{n\geq 0}$ and $(x_n)_{n<0}$ are substitutive.}

	The references provided in claims \eqref{lem:closure2}--\eqref{lem:closure5} all treat one-sided sequences. However, the claims for two-sided sequences follow directly from the corresponding statements for one-sided sequences and \eqref{lem:closure1}.
\end{proof}

\subsection{Topological dynamics}\label{sec:dynamics}

\textit{A (topological) dynamical system} is a compact metrisable space $X$ together with a continuous map $T\colon X\rightarrow X$. If $T$ is a homeomorphism, we say that $(X,T)$ is an \textit{invertible dynamical system}. A point $x\in X$ is \textit{periodic} if $T^k(x)=x$ for some $k\geq 1$.  A dynamical system is called \textit{aperiodic} if it does not contain any periodic points. 
	If $T\colon X\to X$ is a homeomorphism, we let $\mathcal{O}(x)=\{T^n(x)\mid n\in \Z\}$ denote the two-sided orbit of a point $x\in X$.

A \textit{subsystem} of $X$ is a closed subset $Y$ of $X$ that is invariant under the map $T$, i.e.\ $T(Y)\subset Y$. A system $X$ is \textit{minimal} if $X \neq \emptyset$ and if $X$ has no subsystems other than $X$ and the empty set. Equivalently, a system $X\neq \emptyset$ is minimal if the orbit of every point is dense in $X$ \cite[Prop.\ 2.7]{bookKurka}. An invertible system $X$ is  \textit{transitive} if there exists $x\in X$ such that $\overline{\mathcal{O}(x)}=X$.

A dynamical system $(Y,S)$ is a \textit{(topological) factor} of the system $(X,T)$ if there exists a continuous surjective map $\pi\colon X\rightarrow Y$  such that $\pi\circ T=S\circ \pi$. Such a map $\pi$ is called \textit{a factor map}. Two dynamical systems  $(X,T)$ and $(Y,S)$ are \textit{conjugate} (or \textit{isomorphic}) if there exists a homeomorphism $\pi\colon X\rightarrow Y$  such that $\pi\circ T=S\circ \pi$; in this case we call the map $\pi$ a \emph{conjugacy}.

A dynamical system $T\colon X\rightarrow X$ is \textit{equicontinuous} if the family of maps 
\[\{T^n\colon X \rightarrow X \mid n\geq 0\}\]
 is equicontinuous. By classical results,  every dynamical system has a maximal equicontinuous factor \cite[Thm.\ 2.44]{bookKurka} and  every minimal equicontinuous system is conjugate to a translation on a compact abelian group \cite[Thm.\ 2.42]{bookKurka}.

	In this paper we will be mainly concerned with symbolic systems coming from finitely-valued sequences.
	For an alphabet $\mathcal{A}$, the set $\mathcal{A}^{\mathbf{Z}}$ with the product topology (where we use discrete topology on each copy of $\mathcal{A}$) is a compact metrisable space.
	The dynamics on $\A^{\Z}$ is given by the \emph{shift map} 
	\[T\colon \mathcal{A}^{\mathbf{Z}}\rightarrow \mathcal{A}^{\mathbf{Z}} \quad T((x_n)_n)= (x_{n+1})_n,
	\]
	 which acts by shifting the sequence one index to the left. 	
	 	The space $\mathcal{A}^{\mathbf{Z}}$ together with the shift map $T$ is an invertible dynamical system. We refer to subsystems of $\mathcal{A}^{\mathbf{Z}}$ as \textit{subshifts}.

For a subshift $X$, we let
\[\mathcal{L}(X) = \bigcup\{\mathcal{L}(x)\mid x\in X\} \quad \text{and}\quad \mathcal{L}^r(X)=\bigcup\{\mathcal{L}^r(x)\mid x\in X\}
\]  denote the \textit{language} and $r$-\emph{language} of $X$.
	Each subshift is uniquely determined by its language \cite[Prop.\ 3.17]{bookKurka}; in particular, two subshifts are equal if and only if their languages coincide. The above definitions can be adapted in a straightforward way for one-sided  subshifts $X\subset \mathcal{A}^{\N}$.

\subsection{Substitutive systems}\label{sec:substsystems}
 Let $\varphi\colon\mathcal{A}\rightarrow\mathcal{A}^*$ be a substitution. We let
\[X_{\varphi}=\{x\in\mathcal{A}^{\mathbf{Z}}\mid \textrm{ every factor of } x \textrm{ appears in } \varphi^n(a) \textrm{ for some } a\in\mathcal{A},\ n\geq 0\}\]
denote the system generated by $\varphi$.   Note that a two-sided fixed point of $\varphi$ need not lie in $X_{\varphi}$; indeed a point ${}^{\omega}\!\varphi(b).\varphi^{\omega}(a)$ lies in $X_{\varphi}$ if and only if $ba\in\mathcal{L}(X_{\varphi})$.  A  system $X_{\varphi}$  is minimal if and only if $\varphi$ is \emph{primitive}, that is, there is $n\in\N$ such that for all $a,b\in\A$, $a$ appears in $\varphi^n(b)$ \cite[Prop.\ 5.5]{Queffelec-book}.

\begin{remark}\label{rem:languageofsubsystems}
	In general, the set of letters $\mathcal{L}^1(X_{\varphi})$ appearing in the language of $X_{\varphi}$ may be different than $\A$, consider e.g.\ the substitution
	\[\varphi(0)=12,\quad \varphi(1)=22,\quad \varphi(2)=11.
\] on the alphabet $\A=\{0,1,2\}$.
	Thus, in general,
	\[\mathcal{L}(X_{\varphi})\neq \mathcal{L}(\{ \varphi^n(a)\mid a\in \mathcal{L}^1(X_{\varphi}),\ n\geq 0\},
	\]
	where for a set of words $W$, $\mathcal{L}(W)$ denotes the set of all factors of $w\in W$.
	However, due to the fact that $\varphi$ is growing, we always have
	\[\mathcal{L}(X_{\varphi}) = \mathcal{L}(\{ \varphi^n(ab)\mid ab\in \mathcal{L}^2(X_{\varphi}),\ n\geq 0\}.
	\]
\end{remark}

\begin{remark}\label{rem:whenpowersystemsagree}
	Clearly,  $X_{\varphi^{n}}\subset X_{\varphi}$ for all $n\geq 1$. In general, the equality does not need to hold, consider e.g. the system generated by the substitution $\varphi$ in Example \ref{rem:languageofsubsystems} (see also \cite[Rem.\ 1.5]{BKK}).
	However one has the equality under some very general assumption e.g.\ when  $X_{\varphi}$ is transitive \cite[Lemma 2.12]{BKK} or when each letter from $\A$ appears in $\varphi(b)$ for some $b\in \A$ \cite[Lem.\ 5.3]{BPR23}.
\end{remark}

	We will call a subshift $X$ \emph{substitutive} (resp.\ \emph{$k$-automatic}) if it is of the form $X=\tau(X_{\varphi})$ for some  substitution (resp.\ substitution  of length $k$) $\varphi\colon\mathcal{A}\rightarrow\mathcal{A}^*$ and coding $\tau\colon \A\to\B$.
	If one can furthermore take the coding $\tau$ to be the identity, then we will call $X$ \emph{purely substitutive} (resp.\ \emph{purely $k$-automatic}). 
	Note that the coding $\tau$ here defines a factor map from $X_{\varphi}$ to $X$.
	
	It is straightforward to see that the orbit closure $\overline{\mathcal{O}(x)}$ of any substitutive (resp.\ $k$-automatic) sequence $x$ is a substitutive (resp.\ $k$-automatic) system given by the same substitution and coding that define the sequence $x$.
	Conversely, if a substitutive (resp.\ $k$-automatic) system is transitive, then it arises as an orbit closure of a substitutive (resp.\ $k$-automatic) sequence \cite[Lem.\ 2.10]{BKK}.

\subsection{Subsystems of substitutive systems}\label{sec:subsystems}
Let $\varphi\colon \mathcal{A}\rightarrow\mathcal{A}^*$ be a substitution. For $b\in\mathcal{A}$ let $\mathcal{A}_b$ denote the set of letters that appear in $\varphi^n(b)$ for some $n\geq 0$. Note that $\varphi$ maps $\mathcal{A}_b$ to $\mathcal{A}^*_b$. Let $X_{\varphi,b}$ denote the subsystem of $X_{\varphi}$ generated by the substitution $\varphi|_{\mathcal{A}_{b}}\colon \mathcal{A}_b\to\mathcal{A}^*_b$, that is,
 \[X_{\varphi,b}=\{x\in\mathcal{A}^{\mathbf{Z}}\mid \textrm{ every factor of } x \textrm{ appears in } \varphi^n(b) \textrm{ for some }  n\geq 0\}.\]

Theorem \ref{thm:subsystems}  below gathers the results about subsystems of substitutive systems that we will need.

\begin{theorem}\label{thm:subsystems}  Let $\varphi\colon\mathcal{A}\rightarrow\mathcal{A}^*$ be an ambi-idempotent substitution and let $\tau\colon\A\to\B$ be a coding.
\begin{enumerate}
\item\label{thm:subsystems3} \cite[Prop.\ 2.13]{BKK}  For each transitive subsystem $Y\subset X_{\varphi}$ either  $Y=X_{\varphi,b}$ for some $b\in\mathcal{A}$, or \[Y=X_{\varphi,b}\cup X_{\varphi,a}\cup \mathcal{O}({}^{\omega}\!\varphi(b).\varphi^{\omega}(a))\] for some right-prolongable letters $a$ and left-prolongable letter $b$ such that $ba\in\mathcal{L}(X_{\varphi})$. In particular,  each subsystem $Y$ of $X_{\varphi}$ is closed under $\varphi$.
\item\label{thm:subsystems2}\cite[Prop.\ 2.2 and Lem.\ 1.1]{BKK}  All minimal subsystems of $\tau(X_{\varphi})$ are given by $\tau(X_{\varphi,b})$ for some $b\in\mathcal{A}$ such that \[\varphi'=\varphi|_{\mathcal{A}_{b}}\colon\mathcal{A}_{b}\to(\mathcal{A}_{b})^*\] is a primitive substitution. In particular, if  $\varphi$ is of constant length $k$, then $\varphi'$ is of constant length $k$. 
\end{enumerate}
\end{theorem}
\begin{proof} 
	The first claim in \eqref{thm:subsystems3} is shown in \cite[Prop.\ 2.13]{BKK} under the assumption that $\varphi$ is idempotent, which is a somewhat different property than ambi-idempotency (see \cite[Def.\ 1.6]{BKK}). This is a leftover of the fact that the paper \cite{BKK} deals mainly with one-sided substitutive systems for which the property of idempotency is relevant. However, it is easy to see that the only properties of $\varphi$ that are actually used in the proof are those that constitute ambi-idempotency. Since all systems $X_{\varphi,b}$, $b\in\mathcal{A}$ are closed under $\varphi$, and $\varphi(\mathcal{O}({}^{\omega}\!\varphi(b).\varphi^{\omega}(a)))\subset \mathcal{O}({}^{\omega}\!\varphi(b).\varphi^{\omega}(a))$ for any right-prolongable $a$ and left-prolongable $b$,  all transitive subsystems of $X_{\varphi}$ are closed under $\varphi$. Hence, all subsystems of $X_{\varphi}$ are closed under $\varphi$.

The reference \cite[Prop.\ 2.2 and Lem.\ 1.1]{BKK} treats one-sided systems, but the proof in the two-sided case is exactly the same.
\end{proof}

\subsection{$r$-block substitutions}\label{sec:factors}

A common operation in symbolic dynamics consists of looking at  blocks of consecutive symbols and treating them  as the letters of a new, often larger, alphabet. This is done by considering the so-called higher block presentation systems.
For a sequence or a finite word $x$  and integers $i\leq j$ we write  $x_{[i,\,j)}$ for the word $x_ix_{i+1}\cdots x_{j-1}$, if this makes sense; in particular, $x_{[i,i)}=\epsilon$.
	 For a fixed alphabet $\mathcal{A}$ and integer $r\geq 1$, we define the $r$th \emph{higher block presentation} map by the formula
\begin{equation}\label{eq:block_representation}
\iota_r\colon \mathcal{A}^{\Z}\ni y\mapsto \hat{y}\in(\mathcal{A}^r)^{\Z} \quad (\hat{y}_i)_{i\in\mathbf{Z}}= (y_{[i,i+r)})_{i\in\mathbf{Z}},
\end{equation}
where $\mathcal{A}^r$ stands for the set of words of length $r$ over $\mathcal{A}$. 
	For any subshift $X\subset \mathcal{A}^{\Z}$, the map $\iota_r$ defines a conjugacy between $X$ and its $r$th \textit{higher block presentation system} $\iota_r(X)$ which is a subshift over the alphabet $\A^r$ (see \cite[Sec.\ 1.4]{bookMarcusLind} for details). We also use the same symbol to denote the  map
\begin{equation}\label{eq:block_representationwords}
\iota_r\colon \mathcal{A}^{*}\ni w\mapsto \hat{w}\in(\mathcal{A}^r)^{*} \quad  \hat{w}= w_{[0,r)}w_{[1,1+r)}\dots w_{[t-r,t)} 
\end{equation}
between words on the alphabets $\mathcal{A}$ and $\mathcal{A}_r$, where $t=\abs{w}$. Note that if $\abs{w}<r$, then $\iota_r(w)$ is the empty word. If $\abs{w}\geq r$, then $\iota_r(w)$ is a word of length $t-r+1$ over the alphabet $\mathcal{A}^r$.

The following is a standard construction, see e.g.\  \cite[Sect.\ 5.4.1]{Queffelec-book} or \cite[Sect.\ 1.4.5]{bookDurandPerrin}. It allows to express the higher block presentations of substitutive systems again as  substitutive systems (see Proposition \ref{prop:higherblocksubstitutionproperties} below).

\begin{definition}\label{def:higherblocksubstitution} Let $\varphi\colon \mathcal{A}\to\mathcal{A}^*$ be a substitution. Let $r\geq 1$ and let  
\[\widehat{\mathcal{A}_{r}}=\mathcal{L}_r(X_{\varphi})\subset \mathcal{A}^r\]
be the alphabet consisting of all words of length $r$ lying in the language of $\varphi$.  The $r$-\textit{block substitution}  induced by $\varphi$ is the substitution given by 
\[ \hat{\varphi}_r\colon \widehat{\mathcal{A}_{r}}\to\left(\widehat{\mathcal{A}_{r}}\right)^* \quad \hat{\varphi}_r(w)=\varphi(w)_{[0,r)}\ldots \varphi(w)_{[t-1,t+r-1)},
\]
where $t=|\varphi(w_0)|$,  i.e.\ $\hat{\varphi}_r(w)$ is the ordered list of the first $|\varphi(w_0)|$ factors of $\varphi(w)$ of length $r$.
\end{definition}

It is straightforward to check that for any substitution $\varphi\colon\mathcal{A}\to\mathcal{A}^*$ we have  \begin{equation}\label{eq:iotacommutes}
\iota_r(\varphi^n(x))=\hat{\varphi}_r^n(\iota_r(x))\quad \text{for}\quad x\in \mathcal{A}^{\Z}\text{ and } n\geq1,
\end{equation}
see also \cite[pages 49-50]{bookDurandPerrin}. 

The following proposition gathers the  properties of the $r$-block substitution. 

\begin{proposition}\label{prop:higherblocksubstitutionproperties}
Let $r\geq 1$. Let $\varphi\colon \mathcal{A}\to\mathcal{A}^*$ be a substitution and let $\hat{\varphi}_r$ be the induced $r$-block substitution. The following hold:\begin{enumerate}
\item\label{prop:higherblocksubstitutionproperties4} The system generated by  $\hat{\varphi}_r$ is the $r$th higher block presentation of $X_{\varphi}$.
\item\label{prop:higherblocksubstitutionproperties5} If $\varphi$ is primitive, then $\hat{\varphi}_r$ is primitive.
 If $\varphi$ is of constant length $k$, then  $\hat{\varphi}_r$ is of constant length $k$. 
\end{enumerate}
\end{proposition}
\begin{proof}
The second claim follow very easily from the definition of the $r$-block substitution (see also \cite[Lem.\ 5.2]{Queffelec-book} and \cite[Lem.\ 5.3]{Queffelec-book}). Claim \eqref{prop:higherblocksubstitutionproperties4} is standard for primitive $\varphi$; for completeness we provide a proof in the general case. Since a subshift is uniquely determined by its language, to show  \eqref{prop:higherblocksubstitutionproperties4} it is enough to show that $\mathcal{L}(\iota_r(X_{\varphi}))=\mathcal{L}(X_{\hat{\varphi}_r})$. 
	For $r=1$,  $\iota_1(X_{\varphi})=X_{\hat{\varphi}_1}=X_{\varphi}$, so we can assume that $r\geq 2$.
	We write $\iota=\iota_r$ for an ease of notation.

Let $\hat{w}\in \mathcal{L}(\iota(X_{\varphi}))$. There exists $w\in \mathcal{L}(X_{\varphi})$ of length $|w|=|\hat{w}|+r-1$ such that $\hat{w}$ is the ordered list of factors of $w$ of length $r$. By Remark \ref{rem:languageofsubsystems}, there exist $ab\in\mathcal{L}^2(X_{\varphi})$ and $n\geq 1$ such that $w$ appears in $\varphi^n(ab)$. 
	Then $\hat{w}$ appears in $\hat{\varphi}_r^n(v)$ for any $v$ in $\mathcal{L}^r(X_{\varphi})=\widehat{\mathcal{A}_r}$ which starts with $ab$ and thus $\hat{w}$ lies in $\mathcal{L}(X_{\hat{\varphi}_r})$. 

Conversely, let $\hat{w}$ be a word in $\mathcal{L}(X_{\hat{\varphi}_r})$.  There exist $v\in \widehat{\mathcal{A}_r}$ and $n\geq 1$ such that $\hat{w}$ appears in $\hat{\varphi}_r^n(v)$. In particular, there exists a word $w\in \mathcal{A}^*$ of length $|w|=|\hat{w}|+r-1$ such that $\hat{w}$ is the ordered list of factors of $w$ of length $r$, and $w$ appears in $\varphi^n(v)$. Since $v\in \mathcal{L}^r(X_{\varphi})$,  there exist $a\in\mathcal{A}$ and $m\geq 1$ such that $v$ appears in $\varphi^m(a)$. Hence $w$ appears in $\varphi^{n+m}(a)$ and $w$ lies in  $\mathcal{L}(X_{\varphi})$. Thus  $\hat{w}$ lies in  $\mathcal{L}(\iota(X_{\varphi}))$.
\end{proof}

The following well-known result gives a description of factor maps between subshifts; it is sometimes referred to as the Curtis--Hedlund--Lyndon Theorem as it is a special case of a theorem in cellular automata known by this name. 

\begin{theorem}\label{thm:CurtisHedlundLyndon}\cite[Thm.\ 6.2.9]{bookMarcusLind} Let $\pi\colon X\to Y$ be a  map between  subshifts $X$ and $Y$  over alphabets $\mathcal{A}$ and $\mathcal{B}$, respectively. 
The map $\pi$ is a factor map if and only if  there exist an odd integer $r=2n+1$, $n\geq 0$ and a coding 
\[\pi_r\colon \widehat{\mathcal{A}_{r}}\to\mathcal{B}\] such that the following diagram
\begin{equation}\label{eq:diagramCurtis}
  \begin{tikzcd}
    X \arrow{r}{\iota_r} \arrow{d}{\pi} & \iota_r(X) \arrow{d}{\pi_r}   \\
   Y\arrow{r}{T^n} &  Y
  \end{tikzcd}
\end{equation}
commutes, where $\widehat{\mathcal{A}_{r}}=\mathcal{L}_r(X)$ and $\iota_r$ is the $r$th higher block presentation map. 
\end{theorem}

\begin{remark}
	By Proposition \ref{prop:higherblocksubstitutionproperties}\eqref{prop:higherblocksubstitutionproperties4}, one can put $X=X_{\varphi}$ and $\iota_r(X_{\varphi}) = X_{\hat{\varphi}_{r}}$ in diagram \eqref{eq:diagramCurtis}.
	This often allows to transform questions about factor maps from the substitutive system $X_{\varphi}$ to questions about codings of substitutive systems $X_{\hat{\varphi}_{r}}$ defined over larger alphabets $\widehat{\mathcal{A}_{r}}$.
\end{remark}

\begin{remark}
	In Theorem \ref{thm:CurtisHedlundLyndon} for a given factor map $\pi\colon X\to Y$ one can instead work with the shifted higher block presentation map $i'_r =i_r\circ T^n$ so that $\pi = \pi_r\circ \iota'_r$.
	 Note however that with this definition  the equation \eqref{eq:iotacommutes} does not hold as long as we keep the usual definition of $\hat{\varphi_r}$.
	 One can---with a somewhat more complicated construction--- define a substitution $\hat{\varphi'_r}$ on $\widehat{\A_r}$ so that $\iota'_r\circ \varphi = \hat{\varphi'_r}\circ \iota'_r$, see
 \cite[Sect.\ 4.2]{DL-2018} for details; however we will not need it.
\end{remark}

\subsection{Recognisability of substitutive systems}\label{sec:recog} 

Let $X_{\varphi}$ be a purely substitutive system generated by a substitution $\varphi$. It is not difficult to show that each point $x\in X_{\varphi}$ can be 'desubstituted' within $X_{\varphi}$, that is, there exist $x'\in X_{\varphi}$ and $c\in\N$ such that
\begin{equation}\label{eq:desubstitution} 
 x=T^{c}\varphi(x')\quad \text{and}\quad 0\leq c<\abs{\varphi(x'_0)}.
\end{equation}
Note that for a nonperiodic $x\in X_{\varphi}$, any point $x'$  satisfying \eqref{eq:desubstitution} is also nonperiodic.

	One of the most fundamental properties of purely substitutive systems is the so-called \emph{recognizability} which  deals with situations when the desubstitution representation \eqref{eq:desubstitution} of all nonperiodic points $x\in X_{\varphi}$ is unique; the assumption of nonperiodicity is easily seen to be necessary \cite[Rem.\ 2.2(4)]{BSTY}.
	 The notion of recognizability has a long history;  we refer to \cite{BSTY} for a thorough introduction.

	The problem of recognisability of purely substitutive systems  has been  settled in  full generality by Berth\'e et al.\ in \cite[Thm.\ 5.3]{BSTY}: For any substitution $\varphi$  all nonperiodic points $x\in X _{\varphi}$ are \textit{recognisable}, that is, for any nonperiodic $x\in X _{\varphi}$ the representation \eqref{eq:desubstitution} exists and is unique.  This result builds on earlier works \cite{BKM,DM, Mosse-96}, see also \cite{BPR23, BPRS2023}  for a new simplified proof.
	This allows us to make the following definitions.

\begin{definition}\label{desubstitution} Let $\varphi\colon\mathcal{A}\to \mathcal{A}^*$ be a substitution. Let $X'_{\varphi}\subset X_{\varphi}$ denote the set of nonperiodic points in $X_{\varphi}$.
	 Define the `desubstitution' maps
\[ R_{\varphi}\colon X'_{\varphi}\to X'_{\varphi}, \  R_{\varphi}(x)=x' \quad  \text{and} \quad c_{\varphi}\colon X'_{\varphi}\to \N, \ c_{\varphi}(x)= c,
\] to be  unique maps such that $x=T^c(\varphi(x'))$ and $0\leq c< |\varphi(x'_0)|$. For a nonperiodic $x\in X_{\varphi}$, let $\bm{R_{\varphi}}(x)=(R^i_{\varphi}(x))_{i\geq 0}$\footnote{By convention, $R^0_{\varphi}(x)=x$.} and $\bm{c_{\varphi}}(x)=(c_{\varphi}(R^i_{\varphi}(x)))_{i\geq 0}$ denote the  sequences over $X'_{\varphi}$ and $\N$, respectively, obtained by iterating the map $R_{\varphi}$ and then applying $c_{\varphi}$ in the case of the sequence $\bm{c_{\varphi}}(x)$.
\end{definition}

The map $R_{\varphi}$ can be thought of as the 'inverse' of $\varphi$ on $X'_{\varphi}$:  for all $x\in X'_{\varphi}$ we have \[R_{\varphi}(\varphi(x))=x\quad \text{and}\quad \varphi(R_{\varphi}(x))=T^{-c_{\varphi}(x)}(x).\] Note also that $R^m_{\varphi}(x)=R_{\varphi^{m}}(x)$ for all $x\in X'_{\varphi^{m}}\subset X'_{\varphi}$ and all $m\geq 1$.

\begin{remark}\label{rem:desubstitution}
If $\varphi$ is of constant length $k\geq 2$, then the sequence $\bm{c_{\varphi}}(x)=(c_i)_i$ lies in $[0,k-1]^{\N}$. Since $\varphi(T^cx)=T^{ck}\varphi(x)$ for any $c\in\N$, we have
\begin{equation}\label{eq:desubstitutionautomatic}
x=T^{s_{n}}\varphi^n(x^n),\quad \text{where}\quad s_n=\sum_{i=0}^{n-1} c_ik^i.
\end{equation} 
\end{remark}

In general, recognizability fails for one-sided substitutive systems, even the primitive ones (see \cite[Rem.\ 2.2(5)]{BSTY} for an example). 

\begin{remark}\label{rem:finitecphi}
	It is not difficult to see that for any substitution $\varphi\colon\A\to\A^*$ and  sequence $\bm{c}$, there are at most $\abs{\A}^3$ different nonperiodic sequences $x\in X_{\varphi}$ such that $\bm{c_{\varphi}}(x)=\bm{c}$.
	 Indeed, suppose that there are more than $\abs{\A}^3$ nonperiodic points in $X_{\varphi}$ whose $\bm{c_{\varphi}}$ desubstitution sequence is equal to $\bm{c}$.
	  By pigeonhole principle, we can  find two different nonperiodic sequences $x$ and $y$ such that for infinitely many $i\in\N$, 
\[ x^i_{-1}x^i_{0}x^i_1 = y^i_{-1}y^i_{0}y^i_1,
\]
where $x^i=\bm{R_{\varphi}}(x)$ and $y^i=\bm{R_{\varphi}}(y)$. 	Since $\bm{c_{\varphi}}(x)=\bm{c_{\varphi}}(y)$ and $\varphi$ is growing, we have $x=y$, which is a contradiction.
\end{remark}

Below we show that for any factor map from a $k$-automatic system onto some subshift, its fibres on nonperiodic points are uniformly bounded in size. This is well-known in the case of minimal substitutive systems \cite[Thm.\ 20]{D-2000}; we are not aware of a reference in the general case. We will deduce it from the so-called Critical Factorisation Lemma, a classical result in combinatorics of words \cite[Chapter 8]{Lothaire-book1}. 

Let $\mathcal{A}$ be an alphabet and let $W\subset \mathcal{A}^+$ be a finite set. For a two-sided sequence $x$, a $W$-\emph{factorisation} of $x$ is a map $F\colon \Z\to W\times\Z$ such that for all $k\in \Z$, if $F(k)=(w,i)$ and $F(k+1)=(v,j)$, then $x_{[i,j)}=w$; in particular, $x$ can be written as a concatenation of words in $W$. The set of cuts of a $W$-factorisation $F$ of $x$ is the set of all starting positions of factors $w\in W$ in $x$, i.e.\ it is the set $\pi_2(F(\Z))\subset \Z$, where $\pi_2$ denotes the projection on the second coordinate.  Two $W$-factorisations of a  sequence   are \emph{disjoint} if their sets of cuts are disjoint.

Critical Factorisation Theorem (or one of its corollaries) states that a nonperiodic sequence $x\in\mathcal{A}^{\Z}$ has at most $|W|$ pairwise disjoint $W$-factorisations   \cite[Cor.\ 1]{KM-2002}.

\begin{lemma}\label{lem:finitely_many_representations} 
Let $\varphi$ be a substitution of constant length and let  $\pi\colon X_{\varphi}\to Y$ be a factor map onto some subshift $Y$. Then, $|\pi^{-1}(y)|\leq K$ for all nonperiodic $y\in Y$ with some constant $K$ independent of $y$.
\end{lemma}
\begin{proof}
	Using Curtis--Hedlund--Lyndon Theorem (Theorem \ref{thm:CurtisHedlundLyndon}) and Proposition \ref{prop:higherblocksubstitutionproperties}\eqref{prop:higherblocksubstitutionproperties4}, we may assume without loss of generality that the factor map $\pi\colon X_{\varphi}\to Y$ is a coding by passing to some higher block presentation of $X_{\varphi}$.
	We write $\pi=\tau$ in this case.
	
	Let $y\in Y=\tau(X_{\varphi})$ be nonperiodic.
	By Remark \ref{rem:finitecphi}, it is enough to show that the set 	
	 \[\{\bm{c_{\varphi}}(x)\mid x\in X_{\varphi} \text{ and } \tau(x)=y\}\] is finite; we claim that its size is at most $|\mathcal{A}|$. 
 
	Suppose that the claim does not hold and let $F\subset X_{\varphi}$ be the set of size $\abs{A}+1$ such that $\bm{c_{\varphi}}(x)$ are pairwise distinct and $\tau(x)=y$ for all $x\in F$.
	Let $m\geq 1$ be such that the prefixes $c_0^x\dots c_{m-1}^x$ of length $m$ of the sequences $\bm{c_{\varphi}}(x)=(c^x_n)_n\in [0,k-1]^{\mathbf{N}}$ are pairwise different. 
	Note that this implies that all integers 
	\[ s_x=\sum_{n=0}^{m-1}c^x_n k^n,\quad x\in F\]
	 are pairwise different and, by Remark \ref{rem:finitecphi}, we have
	 \begin{equation}\label{eq:cuts}
y=T^{s_{x}}(\tau\varphi^m(x)),	\quad x\in F.		
	\end{equation} 
	 Put $W=\{\tau(\varphi^m(a))\mid a\in\mathcal{A}\}$, and note that $|W|=|\mathcal{A}|$. 
	 Since the morphism $\tau\circ\varphi^m$ is of constant length and $s_x$, $x\in F$, are pairwise different, equation \eqref{eq:cuts} implies that $y$ has $|W|+1$ pairwise disjoint $W$-factorisations. Since $y$ is nonperiodic, this contradicts the Critical Factorisation Theorem  \cite[Cor.\ 1]{KM-2002}.
\end{proof}

\subsection{k-adic integers and $k$-automatic systems}\label{sec:k-adic}

 Let $k\geq 2$ be an integer. 
 	Let  $\mathbf{Z}_k$ denote the ring of $k$-adic integers, that is, the inverse limit $\mathbf{Z}_k=\varprojlim\mathbf{Z}/k^n\mathbf{Z}$ of the inverse system of rings $(\Z/k^i\Z )$, $i\geq 1$ (where the morphisms  are given by the natural quotient maps, and we consider each ring $(\Z/k^i\Z )$ with the discrete topology). 
The ring $\Z_k$ is a compact topological ring; we may identify $\Z_k$ with the following closed subgroup of the direct product of $\Z/k^i\Z$:
\[\Z_k=\{(s_i)_i\in \prod \Z/k^i\Z\mid s_i\equiv s_{i+1} \bmod k^i\}.
\]
As a topological ring $\Z_k$ admits a natural isomorphism
 \begin{equation}\label{eq:isomorphismkadic}\Z_k=\prod \Z_p\quad (z_i)_i\to \prod (z_i \bmod p^i)_i
 \end{equation}
where $p$ runs over the set of primes dividing $k$. The map $f\colon\Z\to\Z_k$ that sends an integer $z$ to the sequence $(z\bmod k^i)_i$ in $\Z_k$ is an injective ring homomorphism, and we identify $\Z$ with the subring $f(\Z)$ of $\Z_k$.

	 Each $k$-adic integer $z$ has a unique $k$-adic expansion $z=.c_0c_1\ldots$ with $c_i\in [0,k-1]$, i.e.\ it can be written uniquely in the form
\begin{equation}\label{eq:kadicsexpan}
z=c_0+c_1k+ c_2k^2+\dots.
\end{equation}
	Furthermore, an integer $z$ is invertible in $\Z_k$ if and only if $\gcd(z,k)=1$, and the set of rational numbers $p/q$, $\gcd(q,k)=1$ in $\Z_k$ corresponds to the set of $k$-adic integers with ultimately periodic $k$-adic expansion.
(Both claims are  classical for the ring of $p$-adic integers $\Z_p$ with $p$ prime \cite[Prop.\ 4.2.2 and Cor.\ 4.3.3]{p-adicbook} and the claim for general $k$ follows from the isomorphism \eqref{eq:isomorphismkadic}.) 

\begin{remark}\label{rem:rationals}
	If $(q,k)=1$, then there exists some $c\geq 1$ (e.g.\ $c=\varphi(q)$ with $\varphi$ denoting here the Euler's totient function) such that $k^c\equiv 1 \bmod q$.
	Thus, any rational number $p/q$ in $\Z_k$ can be written in the form $m(1-k^c)^{-1}$ for some $m\in\Z$ and $c\geq 1$.
\end{remark}

On $\mathbf{Z}_k$ we consider the map $R\colon z\mapsto z+1$, where 1 denotes the identity element of the ring $\mathbf{Z}_k$. The map $R$ is an isometry, and  $\mathbf{Z}_k$ together with  $R$ forms a minimal equicontinuous system.

One of the consequences of recognisability for  aperiodic purely automatic systems  is that they admit the ring of $k$-adic integers as an equicontinuous topological factor.\footnote{This cannot hold for automatic systems with periodic points---once a dynamical system contains periodic points its maximal equicontinuous factor is always finite.}
 Indeed, let $\varphi$ be of length $k$ and assume  that  $X_{\varphi}$ is is aperiodic. By recognizability of all points in $X_{\varphi}$ and Remark \ref{rem:desubstitution}, for each $x\in X$ and each $n\geq 0$ there exists a unique integer $0\leq s_n<k^n$ such that $x=T^{s_{n}}\varphi^n(x^n)$ for some $x^n\in X_{\varphi}$; clearly, $s_{n+1}\equiv s_{n}\bmod k^n$   and  the map 
\begin{equation}\label{eq:factormaptokadics}
\kappa\colon X_{\varphi}\to \Z_k, \quad x\mapsto (s_n)_n
\end{equation} 
is well-defined.  It is straightforward to see that $\kappa$ is continuous and that
\[ \kappa(T^n(x)) = R^n(\kappa(x))=\kappa(x)+n \quad \text{and}\quad \kappa(\varphi^c(x)) = k^c\cdot\kappa(x) \quad \text{for } x\in X_{\varphi}.
\]
Since $(\Z_k,R)$ is minimal, $\kappa(X_{\varphi})=\Z_k$ and $\kappa$ is a factor map onto  $\Z_k$. Furthermore, substitution $\varphi$ corresponds to multiplication by $k$ in $\Z_k$ via $\kappa$. The sequence $\bm{c_{\vartheta}}(x)$ is the sequence of coefficients in the $k$-adic expansion \eqref{eq:kadicsexpan} of $\kappa(x)$.

\begin{remark}\label{rem:factormapuniqueness}
	The map $\kappa$ in \eqref{eq:factormaptokadics} is the unique factor map which sends $\varphi$-periodic points to	 $0\in\Z_k$. In general,  the set of all factors maps  $X_{\varphi}\to\Z_k$ is given by 
\[\{ R_{z}\circ \kappa\mid z\in \Z_k\},\] 
where $R_z(y)=y+z$ denotes the translation by $z$. In particular, for any factor map $\tilde{\kappa}\colon X_{\varphi}\to \Z_k$, $\tilde{\kappa}(x)=\tilde{\kappa}(y)$ if and only if $\bm{c_{\vartheta}}(x)=\bm{c_{\vartheta}}(y)$.
\end{remark}

	The famous Rudin--Shapiro sequence is a well-known example of a (minimal) automatic sequence which is not purely substitutive (in particular, not purely automatic) \cite[Ex.\ 26]{AlloucheTaxonomy}. It is natural to inquire what happens on the level of dynamical systems.
Indeed, in many cases purely substitutive systems are distinctly easier to study,  e.g.\ due to the recognizability properties they enjoy (see Section \ref{sec:recog}).  
	In the case of automatic systems, a recent result of M\"ullner and Yassawi shows that each  minimal $k$-automatic system is isomorphic with a purely $k$-automatic system, see Theorem \ref{thm:MY} below. Example \ref{ex:automaticsystemnotsubstitutive} below shows that this no longer holds for general automatic systems (even the aperiodic ones).

\begin{theorem}\cite[Thm.\ 5 and Thm.\ 22]{MY-21}\label{thm:MY} Let $\varphi$ be a primitive substitution of  length $k$, let $\tau$ be a coding, and assume that $Y=\tau(X_{\varphi})$ is aperiodic.
	 Then $Y$ is conjugate with some purely $k$-automatic system. In particular, $\Z_k$ is a factor of $Y$; furthermore $\kappa\colon Y\to\Z_k$ is a factor map if and only if there exists a factor map $\kappa'\colon X_{\varphi}\to\Z_k$ such that $\kappa'=\kappa\circ\tau$.
\end{theorem}

\begin{example}\label{ex:automaticsystemnotsubstitutive}
Let $\A$ be an alphabet and let $\varphi\colon\A\to\A^*$ be any primitive substitution of length 5 with a nonperiodic fixed point $x$.
Let $x'=T(x)$ be the shift of $x$; the sequence $x'$ is also $5$-automatic, minimal, and nonperiodic. Hence, there exists an alphabet $\A'$, a primitive substitution $\varphi'\colon\A'\to(\A')^*$ of length 5, a fixed point $x''$ of $\varphi'$ , and a coding $\tau'\colon \A'\to\A$ such that $x'=\tau'(x'')$. Let $\mathcal{B}=\{\clubsuit\}\cup\A\cup\A'$ be a new alphabet, choose any letters $a\in\A$, $a'\in\A'$, and define a new substitution  $\vartheta\colon\mathcal{B}\to\mathcal{B}^*$ by
\[\vartheta(\clubsuit)=\clubsuit\clubsuit a a'\clubsuit,\quad \vartheta|_{\A}=\varphi, \quad \vartheta|_{\A'}=\varphi'\]
and a coding $\tau\colon\mathcal{B}\to\A$ by
\[\tau(\clubsuit)=\clubsuit,\quad \tau|_{\A}=\mathrm{id},\quad \tau_{\A'}=\tau'.\]
Let $X=\tau(X_{\vartheta})$ be the $5$-automatic system generated by $(\vartheta,\tau)$. 
	Note that $X$ is transitive (we have $X= \overline{\mathcal{O}({}^{\omega}\!(\vartheta)(\clubsuit).(\vartheta)^{\omega}(\clubsuit))}$) and aperiodic, both $x'$ and $x''$ lie in $ X_{\vartheta}$, $\tau(x')=\tau(x'')=x'$ and
\[\bm{c_{\vartheta}}(x') = (1,0,0,\dots)\neq (0,0,0,\dots) = \bm{c_{\vartheta}}(x'').\]

We will now show that $X$ is not conjugate to any purely automatic system. Suppose otherwise. Then, since $X$ is aperiodic,  there exists a factor map $\pi\colon X\to\Z_2$. It follows that the map $\tilde{\pi}=\pi\circ\tau$ is a factor map from $X_{\vartheta}$ to $\Z_2$. Since $\tau(x')=\tau(x'')$, we have 
$\tilde{\pi}(x')=\tilde{\pi}(x'')$. However, $\bm{c_{\vartheta}}(x')\neq \bm{c_{\vartheta}}(x'')$, which contradicts Remark \ref{rem:factormapuniqueness}.
\end{example} 	
 	
\section{Quasi-fixed points of  substitutions}\label{sec:quasi-fixed}

\begin{definition}\label{def:generalised_fixed_point} Let $\varphi\colon \mathcal{A}\rightarrow \mathcal{A}^*$ be a substitution. A sequence $z\in \mathcal{A}^{\mathbf{Z}}$ is called a \emph{quasi-fixed point} of $\varphi$ if there exist $m\geq 1$ and $c\in \mathbf{Z}$ such that
\[T^c(\varphi^m(z))=z.\] We say that a quasi-fixed point $z$ has  \emph{period} $m$ (w.r.t.\ $\varphi$) if $T^c(\varphi^m(z))=z$ for some $c\in \Z$. A quasi-fixed point $z$ has  \emph{minimal period} $m$ (w.r.t.\ $\varphi$) if $m$ is a period of $z$ and no $n<m$ is a period of $z$.
\end{definition}

Note that $\varphi$-periodic points (i.e.\ points $z\in\A^{\Z}$ such that $\varphi^m(z)=z$ for some $m\geq 1$) are quasi-fixed points of $\varphi$; furthermore all shifts of quasi-fixed points of $\varphi$ are quasi-fixed points of $\varphi$ (see Proposition \ref{prop:gen_points_are_automatic}\eqref{prop:gen_points_are_automatic1}). It is not difficult to find quasi-fixed points of $\varphi$ which are not shifts of $\varphi$-periodic points, consider e.g.\ Example \ref{ex:generalised_fixed_point} below.

\begin{example}\label{ex:generalised_fixed_point} Let   $\varphi\colon \mathcal{A} \rightarrow \mathcal{A}^*$ be the Thue-Morse substitution $\varphi(0)=01, \varphi(1)=10$ and consider its fourth iterate given by
\begin{equation}\label{eq:example_qp}
\varphi^4(0)=0110100110010110, \quad
\varphi^4(1)=1001011001101001.
\end{equation} Write $\varphi^4(0)=v0w$, where $v=01101$, $w=0110010110$. After iterating \eqref{eq:example_qp} we get:
\[\varphi^8(0)=\varphi^4(v)v0w\varphi^4(w),\quad
\varphi^{12}(0)=\varphi^8(v)\varphi^4(v)v0w\varphi^4(w)\varphi^8(v),\quad \ldots\] The sequence $z=\ldots \varphi^8(v)\varphi^4(v)v.0w\varphi^4(w)\varphi^8(v) \ldots$ lies in $X_{\varphi}$  and $T^5(\varphi^4(z))=z$. The sequence $z$ is not a $\varphi$-periodic point nor its shift.
\end{example}

Shallit and Wang \cite{twosidedfixedpointsSW} (see also \cite[Sec.\ 7.4]{AlloucheShallit-book}) and more recently B\'eal, Perrin, and Restivo \cite{BPR24} studied for a given substitution $\varphi\colon\mathcal{A}\rightarrow\mathcal{A}^*$ the set of points \[\{x\in\mathcal{A}^{\Z}  \text{ such that } x=T^c(\varphi(x)) \text{ for some } c\in \Z\},\] and obtained a complete characterisation of this set. 
	This characterisation  implies that  all  quasi-fixed points of $\varphi$, modulo the shift operation, are either $\varphi$-periodic points, periodic points (w.r.t\ the shift) or are of the special form suggested by Example \ref{ex:generalised_fixed_point} above. 			
	The characterisation in \cite{twosidedfixedpointsSW, BPR24}  does not require the assumption that $\varphi$ is growing. When $\varphi$ is growing, the characterisation simplifies and readily gives rise to the following description of the set of quasi-fixed points of $\varphi$. It will be convenient to adopt the following notation from  \cite[Sec.\ 7.4]{AlloucheShallit-book}. 

\begin{definition}\label{def:quasifixedpointexplicitformula} Let $\varphi\colon\mathcal{A}\to\mathcal{A}^*$ be a substitution. Let $a\in \A$ and assume that 
\begin{equation}\label{eq:defquasifixedexplicit}
\varphi(a)=vav' \quad \text{for some}\quad v,v'\in\A^*.
\end{equation}
Recall that if we can write $\varphi(a)$ in the form \eqref{eq:defquasifixedexplicit} with $v$  empty, then $a$ is called  \emph{right-prolongable} and we use the notation
\[\varphi^{\omega}(a)=av'\varphi(v')\varphi^2(v')\dots\]
If  we can write $\varphi(a)$ in the form \eqref{eq:defquasifixedexplicit} with $v'$ empty, we say that $a$ is \emph{left-prolongable} and define a left-infinite word
\[{}^{\omega}\!\varphi(a) = \dots \varphi^2(v)\varphi(v)a.\] If both $v$ and $v'$ are nonempty,  we  define  the biinfinite sequence
\[\varphi^{\omega,i}(a)=\ldots\varphi^2(v)\varphi(v)v.av'\varphi(v')\varphi^{2}(v')\ldots,\] where $i=|v|$.
\end{definition}

\begin{proposition}\cite[Sec.\ 7.4]{AlloucheShallit-book}\label{prop:char_quasi_periodic_one}  Let $\varphi\colon \mathcal{A}\rightarrow \mathcal{A}^*$ be a substitution and let $x\in\mathcal{A}^{\Z}$. Then $x$  is a quasi-fixed point of $\varphi$ of period $m$ if and only if $x$ is a shift of $y$ that is of one of the following forms:\begin{enumerate}
\item \label{prop:char_quasi_periodic_one1}  $y= {}^{\omega}\!(\varphi^m)(b).(\varphi^m)^{\omega}(a)$ for some  $a,b\in\mathcal{A}$ which are, respectively, right and left-prolongable with respect to $\varphi^m$.
\item \label{prop:char_quasi_periodic_one2}  $y=(\varphi^m)^{\omega,i}(a)$ for some  $a\in\mathcal{A}$.
\end{enumerate} 
For a point $x$ to lie in $X_{\varphi}$, we need to further assume in \eqref{prop:char_quasi_periodic_one1} that  $ba\in \mathcal{L}(X_{\varphi})$. 
\end{proposition}
\begin{proof} Since all quasi-fixed points of $\varphi$ of period $m$ are quasi-fixed points of $\varphi^m$ of period $1$, we may assume without loss of generality that $m=1$. The first claim now follows from \cite[Thm.\ 7.4.3]{AlloucheShallit-book}. To see it, note that that the cases $(a)-(d)$ of  \cite[Thm.\ 7.4.3]{AlloucheShallit-book} all describe shifts of fixed points of $\varphi$, cf.\ Proposition  \cite[Prop.\ 7.4.1]{AlloucheShallit-book}, and that the set $F_{\varphi}$  is empty when $\varphi$ is growing. Since $\varphi$ is growing, the case $(f)$ cannot occur. (Indeed, suppose that there exist nonempty finite words $u,v$ such that $\varphi(uv)=vu$. Then an easy induction shows that $|\varphi^n(uv)|=|v|+|u|$ for all $n\geq 1$, which contradicts the fact that $\varphi$ is growing.) The characterisation of quasi-fixed points which lie in $X_{\varphi}$ follows easily from the definition of $X_{\varphi}$.
\end{proof}

Let $\varphi\colon\A\to\A^*$ be a substitution of constant length $k\geq 2$. The following definition introduces certain families $\mathrm{F}_{\varphi,m}$ of maps $\Psi\colon \A^{\Z}\to\A^{\Z}$, which allow us to express the $k$-kernels of points in $X_{\varphi}$ in a particularly convenient way (see Lemma \ref{lem:closure}\eqref{lem:closure5} for the definition of the $k$-kernel of a sequence).

\begin{definition}\label{def:kernelmaps}
	For an alphabet $\A$, let  $\F_{\mathcal{A}}$ denote the set of functions $f\colon \mathcal{A}^{\Z}\to\mathcal{A}^{\Z}$ that are induced by some function $f'\colon\mathcal{A}\to\mathcal{A}$, i.e.\ such that $f((x_n)_{n})=(f'(x_n))_{n}$.
	
	For a substitution $\varphi \colon \mathcal{A}\rightarrow \mathcal{A}^*$ of length $k$ and each $i=0,\ldots,k-1$ define the map
\[\Psi_i\colon \mathcal{A} \to  \mathcal{A}\quad a\to \varphi(a)_i\]
that sends a letter $a$ to the $i$th letter of $\varphi(a)$.
	We will  use the same symbol to denote the extension of $\Psi_i$ by concatenation to  $\Psi_i\colon \mathcal{A}^{\Z}\rightarrow \mathcal{A}^{\Z}$.  For $m\geq 1$ and a word $\mathbf{i}=i_0\cdots i_{m-1}\in [0,k-1]^{m}$ put 
\[\Psi_{\mathbf{i}}=\Psi_{i_{m-1}}\circ\cdots\circ\Psi_{i_{0}}\] and let  $\Psi_{\mathbf{i}}$ denote the identity map if $\mathbf{i}$ is the empty word. 		
	
	For each $m\geq 0$ define the family of maps
\[ \mathrm{F}_{\varphi,m}=\{\Psi_{\mathbf{i}}\colon \mathcal{A}^{\Z}\to\mathcal{A}^{\Z}\mid \mathbf{i}\in [0,k-1]^m\}\]
that consists of all possible compositions of length $m$ of functions $\Psi_i$; in particular, $\mathrm{F}_0$ consists of the identity function.
\end{definition}

\begin{lemma}\label{lem:setFisfinite}
For any substitution $\varphi$ of constant length, the union $\F$ of the sets $\F_{\varphi,m}$, $m\geq 0$ is finite.
\end{lemma}
\begin{proof}
 Clearly, the set $\F_{\mathcal{A}}$ is finite and $\F_{\varphi,m}\subset \F_{\mathcal{A}}$ for all $m\geq 0$ and so $F$ is finite.
\end{proof}
	
	 	Note  that $\Psi_{\bm{i}}(T(x))=T(\Psi_{\bm{i}}(x))$ for all $x\in\mathcal{A}^{\Z}$ and $\bm{i}\in [0,k-1]^m$, $m\geq 0$.
	 	Furthermore, for  $\mathbf{i}=i_0\cdots i_{m-1}$ in  $[0,k-1]^{m}$ the map $\Psi_{\mathbf{i}}\colon \mathcal{A}^{\Z}\rightarrow \mathcal{A}^{\Z}$ sends a sequence $x=(x_n)_n$ to the sequence \begin{equation}\label{eq:descriptionofmapsPsi}((\varphi^{m}(x))_{k^{m}n+s_{\mathbf{i}}})_n,\quad\text{where}\quad s_{\mathbf{i}}=\sum_{l=0}^{m-1} i_{m-l-1}k^l. 
\end{equation} 
	In particular, $\F_{\varphi,m}=\F_{\varphi^m,1}$ for each $m\geq 1$ and, if $x\in\mathcal{A}^{\Z}$ is a fixed point of $\varphi$, then its $k$-kernel is given by  
\[\mathrm{K}_k(x)=\{\Psi_{\mathbf{i}}(x)\mid \mathbf{i}\in [0,k-1]^m, m\geq 0\}.\] 
	For other points in $X_{\varphi}$ we have the following lemma.
	
\begin{lemma}\label{lem:kernel} Let $\varphi\colon \mathcal{A}\rightarrow \mathcal{A}^*$ be a substitution of  length $k$. Let $x\in \mathcal{A}^{\mathbf{Z}}$. Let  $(x^i)_{i\geq 0}$ be a sequence of elements of $\mathcal{A}^{\mathbf{Z}}$, and let $(c_i)_{i\geq 0}$ be a sequence of elements of $[0,k-1]$ such that 
\begin{equation}\label{eq:representionlemkernel}x^0=x\quad \text{and}\quad x^i=T^{c_{i}}(\varphi(x^{i+1}))\quad \text{for}\quad i\geq 0.
\end{equation} 
The $k$-kernel of $x$ is given by
\begin{align*}
\mathrm{K}_k(x)  = & \{\Psi_{\mathbf{i}}(x^m) \mid m\geq 0, \mathbf{i}\in [0,k-1]^m, \sum_{l=0}^{m-1} i_{m-l-1}k^l\geq \sum_{l=0}^{m-1} c_lk^l\} \quad \cup \\
& \{\Psi_{\mathbf{i}}(T(x^m)) \mid m\geq 0, \mathbf{i}\in [0,k-1]^m, \sum_{l=0}^{m-1} i_{m-l-1}k^l < \sum_{l=0}^{m-1} c_lk^l\}. 
\end{align*}
\end{lemma}
\begin{proof}
It is enough to note that $x=T^{t_{m}}(\varphi^m(x^m))$, where $t_m=\sum_{l=0}^{m-1}c_lk^l$ for all $m\geq 1$, and that 
\[\Psi_{\mathbf{i}}(x^m)=(\varphi^m(x^m)_{k^mn+ s_{\mathbf{i}}})_n,\] where
\[\quad s_{\mathbf{i}}=\sum_{l=0}^{m-1} i_{m-l-1}k^l\quad \text{for} \quad \mathbf{i}=i_0\dots i_{m-1}\in [0,k-1]^m, \ m\geq 0. \qedhere\] 
\end{proof}

Let $(x_i)_{i\in \N}$ be a sequence over some set. We say that a sequence $(x_i)_i$ is ultimately periodic with period $m$ if there exist integers $p<q$ such that $q-p=m$ and $x_i=x_{i+m}$ for $i\geq p$. We say that the period $m$ of an ultimately periodic sequence $(x_i)_i$ is minimal if no $n<m$ is a period of $(x_i)_i$. If $p$ can be taken to be $0$, we say that $(x_i)_i$ is (purely) periodic with period $m$ (resp.\ minimal period $m$). The following proposition gathers the main closure properties of quasi-fixed points of substitutions.

\begin{proposition}\label{prop:gen_points_are_automatic}  Let $\varphi\colon \mathcal{A}\rightarrow\mathcal{A}^*$ be a substitution. \begin{enumerate}
\item\label{prop:gen_points_are_automatic0} Let $r\geq 1$ and let $\iota_r$ be the $r$th higher block presentation map. Then $x\in\A^{\Z}$ is a quasi-fixed point of $\varphi$ of period $m$ (resp.\ of minimal period $m$) if and only if $\iota_r(x)$  is a quasi-fixed point of $\hat{\varphi_r}$ of period $m$ (resp.\ of minimal period $m$).
\item\label{prop:gen_points_same_for_power} For any $n\geq 1$ the sets of quasi-fixed points of $\varphi$ and $\varphi^n$ coincide.
\item\label{prop:gen_points_are_automatic1}   For each $m\geq 1$, the set of quasi-fixed points of $\varphi$ of period $m$ (resp.\ of minimal period $m$) is closed under the left and right shifts. Furthermore, the set of nonperiodic quasi-fixed points of $\varphi$ in $X_{\varphi}$ of period $m$ (resp.\ of minimal period $m$) is closed under $\varphi$ and $R_{\varphi}$.
\item\label{prop:gen_points_are_automatic3}   Let $z\in X_{\varphi}$ be nonperiodic. Then, $z$ is a quasi-fixed point of $\varphi$ of period $m$  (resp.\ of minimal period $m$) if and only if $\bm{R_{\varphi}}(z)$ is an ultimately periodic sequence of period $m$  (resp.\ of minimal period $m$). In particular,  periods of $z$ as a quasi periodic point of $\varphi$ are exactly the multiples of its minimal period.
\item\label{prop:gen_points_are_automatic2}   Every quasi-fixed point of $\varphi$ is substitutive; furthermore, if $\varphi$ is of  length $k$, then every quasi-fixed point of $\varphi$ is $k$-automatic.
\end{enumerate}
\end{proposition}
\begin{proof}  For $x,y\in \mathcal{A}^{\Z}$, we will write (in this proof only) $x\sim y$ if $x$ and $y$ are shifts of each other, i.e.\ $x=T^c(y)$ for some $c\in\Z$. Note that if $x\sim y$, then $\varphi(x)\sim \varphi(y)$ for any substitution $\varphi\colon \mathcal{A}\to \mathcal{A}^*$. 

The first claim follows directly from the formula
\[\iota_r T^c\varphi^m(z) = T^c\hat{\varphi_r}^m\iota_r(z)
\] with $z\in \A^{\Z}$, $c\in\Z$, and $m\geq 1$.

The second claim follows from the fact that for any quasi-periodic point $z$ of period $m$ and any $n\geq 1$ we have
\[z\sim\varphi^m(z)\sim \varphi^{2m}(z)\sim \dots \sim\varphi^{nm}(z),
\] 
	since $z\sim \varphi^m(z)$ and $\varphi$ preserves the relation $\sim$.
	Since \[\varphi^m(T^n(z))\sim \varphi^m(z) \sim z \sim T^n(z)\] for all $c\in\Z$, the set of quasi-fixed points of period $m$ is preserved under the shift $T^n$ for any $n\in\Z$.
	 Since the set of quasi-fixed points of period $m$ is closed under the shift operation $T$ and its inverse $T^{-1}$, the set of quasi-fixed points of  minimal period $m$ is closed under both $T$ and $T^{-1}$. Since $z\sim \varphi^m(z)$ and $\varphi$ preserves the relation $\sim$ we also have that \[\varphi^m(\varphi(z))=\varphi(\varphi^m(z))\sim \varphi(z),\] and hence the set of quasi-fixed points of period $m$ is also preserved under $\varphi$. Now we will  show  \eqref{prop:gen_points_are_automatic3} and use it to deduce the rest of the claims in \eqref{prop:gen_points_are_automatic1}.  

Let $z\in X_{\varphi}$ be nonperiodic. It is enough to show both implications for period $m$, the claims for minimal period $m$ will then obviously follow. Assume that  $z$ is a quasi-fixed point of $\varphi$ of period $m$.  By Proposition \ref{prop:char_quasi_periodic_one}, there exists $x\in X_{\varphi}$ which is either a fixed point of $\varphi^m$ or is of the form $x=(\varphi^m)^{\omega,i}(a)$ such that $z\in \mathcal{O}(x)$.  We have the following easy claims.
\begin{enumerate}
\item\label{orbit1} If  $x=(\varphi^m)^{\omega,i}(a)$, then $R_{\varphi}^{k}(z)=x$ for some $k\geq 0$.
\item\label{orbit2} If $x$ is a fixed point of $\varphi^m$, then either $R_{\varphi}^{k}(z)=x$ or $R_{\varphi}^{k}(z)=T^{-1}x$  for some $k\geq 0$.
 \end{enumerate}
To see that \eqref{orbit1} holds,  let $a\in\mathcal{A}$ and let $x=(\varphi^m)^{\omega,i}(a)$, i.e.\ $\varphi^m(a)=vav'$ for some nonempty words $v,v'$ with $|v|=i$, and \[x=\ldots\varphi^{2m}(v)\varphi^m(v)v.av'\varphi^m(v')\varphi^{2m}(v')\ldots\] Note that $x_0=a$ and that  $x=T^{i_{k}}(\varphi^{(k+1)m}(x))$ for each $k\geq 0$, where\[0< i_k = |\varphi^{km}(v)\dots\varphi^m(v)v| <|\varphi^{(k+1)m}(a)|.\] Since  both $v$ and $v'$ are nonempty, both $i_k$ and $|\varphi^{(k+1)m}(a)|-i_k$ go to infinity as $k\to\infty$. Hence, every point in  $\mathcal{O}(x)$ can be obtained by shifting $\varphi^{km}(x)$ by some $0\leq c < |\varphi^{km}(x_0)|$ for $k$ big enough. Claim \eqref{orbit2} can be shown in a similar way.  Note that every $y\in X_{\varphi}$, which is a  fixed point of $\varphi^m$, a shift  of a fixed point of $\varphi^m$ by $T^{-1}$, or is of the form $y=(\varphi^m)^{\omega,i}(a)$ satisfies $R^m_{\varphi}(y)=y$, i.e.\ $\bm{R_{\varphi}}(y)$ is a (purely) periodic sequence of period $m$. Hence, by \eqref{orbit1} and \eqref{orbit2}, $\bm{R_{\varphi}}(z)$ is an ultimately periodic sequence of period $m$.

Now assume that $\bm{R_{\varphi}}(z)$ is an ultimately periodic sequence with period $m$. Then, there exist some $k\geq 1$ and a quasi-fixed point $x\in X_{\varphi}$ of period $m$ such that \begin{equation}\label{eq:gen_points} R_{\varphi}^{k}(z)=x.\end{equation}  Applying $\varphi^k$ to the equation \eqref{eq:gen_points} we see that
\[z\sim \varphi^k(R_{\varphi}^{k}(z))=\varphi^k(x).\] Since  the set of quasi-fixed points of period $m$ is closed under $\varphi$ and the shift, $z$ is a quasi-fixed point of $\varphi$ of period $m$. The last claim in \eqref{prop:gen_points_are_automatic3} now follows, since the periods of an ultimately periodic sequence $\bm{R_{\varphi}}(z)$ are exactly the multiples of its minimal period. Since $\bm{R_{\varphi}}(R_{\varphi}(z))$ is equal to $\bm{R_{\varphi}}(z)$ shifted by one and  $\bm{R_{\varphi}}(\varphi(z))$ is equal to $\bm{R_{\varphi}}(z)$ with 0 added at the beginning, the rest of claims in \eqref{prop:gen_points_are_automatic1} follow as well.

	Now we will  show the last claim. We first treat the case of general substitutions.  
	 Since $x$ is substitutive if and only if the one-sided sequences $(x_n)_{n\geq 0}$ and $(x_n)_{n<0}$ are substitutive, and substitutive sequences are preserved under the left and right shift,  by  Proposition  \ref{prop:char_quasi_periodic_one}, it is enough to show that a one-sided sequence $v\varphi^m(v)\varphi^{2m}(v)\ldots$ is substitutive for any word $v$ and any $m\geq 1$.
	This is easy to see (and has been observed several times \cite[Lem.\ 5]{CRSZ} or \cite[Lem.\ 1.2(i)]{BKK}): Let $\spadesuit$ be a new symbol not lying in the alphabet $\A$ and consider a new substitution $\varphi'$ defined on the alphabet $\A'=\A\cup \{\spadesuit\}$ given by 
\[\varphi'(\spadesuit)=\spadesuit v, \quad \varphi'|_{\A}=\varphi^m\] together with the coding $\tau$ which sends $\spadesuit$ to $a$. 
	Then, $x=\tau(x')$, where $x'$ is the fixed point of $\varphi'$ starting with $\spadesuit$.

	Now assume that $\varphi$ is of constant length. In this case, the claim can be deduced from  \cite[Thm.\ 5.3.4]{AlloucheShallit-book} and \cite[Thm.\ 6.7.2]{AlloucheShallit-book}, or \cite[Lem.\ 1.2(ii)]{BKK}. 
	For completeness, we provide a short proof.
	We will use the kernel-based characterisation of automaticity (Lemma \ref{lem:closure}\eqref{lem:closure5}) and the sets $\mathrm{F}_{\varphi,m}$ from Definition \ref{def:kernelmaps}.
	Let $x$ by a quasi-fixed point of $\varphi$ of period $m\geq 1$. 		
	Since the sets of $k$-automatic and $k^m$-automatic sequences coincide, by taking the $m$th power of $\varphi$, we may assume that $x$ is a quasi-fixed point of period 1. 
	By Proposition \ref{prop:char_quasi_periodic_one}, $x$ is a (left or right) shift of some $x'$ such that $x'=T^c\varphi(x')$ for some $0\leq c<k$.
	 Let $\mathrm{F}$ be the union of the sets $\mathrm{F}_{\varphi,m}$, $m\geq 0$ and recall that $\F$ is finite by Lemma \ref{lem:setFisfinite}.
	By Lemma \ref{lem:kernel}, 
\[\mathrm{K}_k(x')\subset \{\Psi(x')\mid \Psi\in \mathrm{F}\}\cup \{\Psi(T(x'))\mid \Psi\in \mathrm{F}\}.\]
Since $\mathrm{F}$ is finite, the $k$-kernel of $x'$ is also finite.
	 Thus, $x'$ is $k$-automatic and so $x$ is $k$-automatic as a shift of a $k$-automatic sequence.
\end{proof}

\section{Automatic sequences in automatic systems}\label{sec:automatic sequences in systems}

In this section we provide a complete characterisation of the set of automatic sequences in an automatic system in terms of the quasi-fixed points of the underlying constant-length substitution.

\begin{theorem} \label{thm:characteraytomaticsequences} Let $k\geq 2$. Let $\varphi$ be a substitution of length $k$. Let $\pi\colon X_{\varphi}\to Y$ be a factor map onto some subshift $Y$. The following hold.\begin{enumerate}
\item \label{thm:characteraytomaticsequences1} A sequence $y\in Y$ is $k$-automatic if and only if $y=\pi(x)$ for some quasi-fixed point $x$ of $\varphi$.
\item \label{thm:characteraytomaticsequences2} If $y\in Y$ is $k$-automatic and nonperiodic, then all  points in $\pi^{-1}(y)$ are quasi-fixed points of $\varphi$.
\end{enumerate}
\end{theorem}

	The assumption that $z\in Y$ is nonperiodic in the last claim is clearly necessary: it can happen that $\tau$ maps an (uncountable) subsystem of $X_{\varphi}$ onto a periodic subsystem of $Y$; for concreteness consider e.g.\ Example \ref{ex:periodic} below.
\begin{example}\label{ex:periodic} Let $\mathcal{A}=\{0,1,2,3\}$ and let $\varphi\colon \mathcal{A}\rightarrow \mathcal{A}^*$ be a substitution given by 
\[\varphi(0)=0123,\quad \varphi(1)=1031,\quad \varphi(2)=2332,\quad \varphi(3)=3223.\] Let $\tau\colon\mathcal{A}\rightarrow \{0,1,2\}$ be a coding identifying the letters $2,3\in\mathcal{A}$, i.e.\ \[\tau(0)=0,\quad \tau(1)=1,\quad \tau(2)=\tau(3)=2.\] The constant sequence $z$ of 2's belongs to $Y=\tau(X_{\varphi})$ and $\tau^{-1}(z)\subset X_{\varphi}$ is the Thue--Morse system generated by $\varphi|_{\{2,3\}}$. Since the Thue--Morse system is uncountable and there are only countably many substitutive sequences, it follows that the set $\tau^{-1}(z)$ contains some sequences that are not substitutive.
\end{example}

In the case of an infinite minimal $k$-automatic system its set of automatic sequences can be characterised in terms of the arithmetic properties of the ring $\Z_k$, its equicontinuos factor (see Section \ref{sec:k-adic} for information about $\Z_k$ and the definition of the factor map to $\Z_k$).

\begin{corollary}\label{cor:automaticcorrespondtorational}
 Let $k\geq 2$.
 	 Let $\varphi$ be a primitive substitution of length $k$, let $\mathrm{P}\subset X_{\varphi}$ be the set of $\varphi$-periodic points, and let $\tau$ be a coding. 
 	 Assume $X=\tau(X_{\varphi})$ is aperiodic and let $\kappa\colon X\to \Z_k$ be the factor map sending $\tau(\mathrm{P})$ to $0\in \Z_k$.
 	Then, $x\in X$ is automatic if and only if $\kappa(x)\in \Z_k$ is rational.
 	For $x=\tau(x')$ with $T^m\varphi^c(x')=x'$, $m\in\Z$, $c\geq 1$ we have $\kappa(x)=m(1-k^c)^{-1}$.
\end{corollary}
\begin{proof}
	By Theorem \ref{thm:MY}, $\kappa' = \kappa\circ\tau$, where $\kappa'\colon X_{\varphi}\to \Z_k$ is the factor map sending the set $\mathrm{P}$ of $\varphi$-periodic points  to $0$.  
	The last claim follows thus from the equality 
\begin{equation}\label{eq:simplearith}
\pi'(x') = \pi'(T^m\varphi^c(x'))=m+ k^c\pi'(x').
\end{equation}
	By Theorem \ref{thm:characteraytomaticsequences}, all automatic sequences in $X$ are of the form $x=\tau(x')$ with $x'\in X_{\varphi}$ the quasi-fixed point of $\varphi$.
	 Thus, the first claim holds, by \eqref{eq:simplearith} and the fact the each rational number $z$ in $\Z_k$ can be written in the form $z=m(1-k^c)^{-1}$ for some $m\in\Z$ and $c\geq 1$ (see Remark \ref{rem:rationals}). 
\end{proof}

We conjecture that Theorem \ref{thm:characteraytomaticsequences} holds  for general substitutive systems (not necessarily of constant length), i.e.\ that the following statement holds.

\begin{conjecture} \label{con:characteraytomaticsequences}  Let $\varphi$ be a substitution. Let $\pi\colon X_{\varphi}\to Y$ be a factor map onto some subshift $Y$. The following hold.
\begin{enumerate}
\item \label{con:characteraytomaticsequences1} A sequence $y\in Y$ is substitutive if and only if $y=\pi(x)$ for some quasi-fixed point $x$ of $\varphi$.
\item \label{con:characteraytomaticsequences2} If $y\in Y$ is substitutive and nonperiodic, then all  points in $\pi^{-1}(y)$ are quasi-fixed points of $\varphi$.
\end{enumerate}
\end{conjecture}
	
\begin{remark}\label{rem:holton}
	By the result of Holton and Zamboni \cite[Thm.\ 6.1 and Lem.\ 6.3]{HZ-01}, Conjecture \ref{con:characteraytomaticsequences} is known in the case when $\varphi$ is primitive and $\pi$ is a coding.   Knowing this one can obtain Conjecture \ref{con:characteraytomaticsequences} for primitive $\varphi$ and general factor map $\pi$ using the results from Section \ref{sec:factors} in the same way they are used in the proof of Theorem \ref{thm:characteraytomaticsequences}.
\end{remark}

\begin{remark}\label{rem:conjugacyproblemsviaquasifixed}
	Let $\varphi$ and $\varphi'$ be any primitive substitutions of the same length $k\geq 2$ with aperiodic systems $X_{\varphi}$ and $X_{\varphi'}$.
	In \cite{Coven2015}, the authors give an algorithm to describe the set of factor maps $\pi\colon X_{\varphi}\to X_{\varphi'}$ (or the automorphism group $\mathrm{Aut}(X_{\varphi})$ in the case $\varphi=\varphi'$) via the detailed study of the "fingerprint" left by the map $\pi$ on the equicontinuous factor $\Z_k$. 
	A crucial step \cite[Thm.\ 3.19 and Prop.\ 3.24]{Coven2015} --- translated to our context via Corollary \ref{cor:automaticcorrespondtorational}---gives a computable bound on the minimal period of the quasi-fixed point of $\varphi'$ to which a $\varphi$-periodic point can be send by a factor map.	
	While in the general case of nonminimal or nonconstant length substitutive systems  we no longer have the underlying structure of $k$-adic integers as an equicontinuous factor, one can still make sense of this approach via quasi-fixed points and their periods.
\end{remark}

The rest of the section is devoted to the proof of Theorem \ref{thm:characteraytomaticsequences}. Let $\varphi\colon \mathcal{A}\rightarrow \mathcal{A}^*$ be a substitution of length $k$. We will make use of the sequences $\Psi_{\bm{i}}$ and families $\mathrm{F}_{\varphi,m}$ defined in Definition \ref{def:kernelmaps}. Here, it will be advantageous to assume that our substitution has the following idempotency property.

\begin{definition}\label{def:pair-constant} Let $\varphi\colon \mathcal{A}\rightarrow \mathcal{A}^*$ be a substitution of constant length. We say that $\varphi$ is \textit{column-constant} if the sets $\F_{\varphi,m}$ are the same for all $m\geq 1$.
\end{definition}

	The name column-constant stems from the fact that, visually, a substitution $\varphi\colon \mathcal{A}\to\mathcal{A}^*$ is column-constant if and only if the sets of columns that appear in the $|\mathcal{A}|\times k^n$ arrays formed by the words $\varphi^n(a)$, $a\in\mathcal{A}$ are the same  for all $n\geq 1$.   
	The following lemma shows that we can always assume that a substitution $\varphi$ is column-constant by passing to some power of $\varphi$.

\begin{lemma}\label{lem:all_pairs} Let $\varphi\colon \mathcal{A}\rightarrow \mathcal{A}^*$ be a substitution of  length $k$. \begin{enumerate}
\item\label{lem:all_pairs1} For any $x,y\in \mathcal{A}^{\Z}$ and $m\geq 1 $, $\varphi^m(x)=\varphi^m(y)$ if and only if $\Theta(x)=\Theta(y)$ for all $\Theta\in \F_{\varphi,m}$.
\item \label{lem:all_pairs2} There exists $n\geq 1$ such that the substitution $\varphi^n$ is column-constant.
\end{enumerate}
\end{lemma}
\begin{proof} 
	The first claim follows easily from the fact that $\mathrm{F}_{\varphi,m}=\mathrm{F}_{\varphi^m,1}$ for any $m\geq 1$.
	To see the second claim, let $\mathcal{P}$ be the set of all nonempty subsets of $\F_{\mathcal{A}}$. Define the function $G\colon \mathcal{P}\to \mathcal{P}$ that sends a set $P\in \mathcal{P}$ to the set \[G(P)=\{\Psi_i\circ f\mid i\in [0,k-1],\ f\in P\}.\]  Note that $G^m(\F_{\varphi,1})=\F_{\varphi,m}$ for all $m\geq 1$. Since $G$ is defined on a finite set,  there exists $n\geq 1$ such that the function $G^n$ is idempotent (e.g.\ by \cite[Lem.\ 1.7]{BKK}). In particular, $G^{n}(\F_{\varphi,1})=G^{mn}(\F_{\varphi,1})$ for all $m\geq 1$, and hence $\varphi^n$ is column-constant. 
\end{proof}

Let $\varphi$ be a column-constant substitution of constant length and let $\tau$ be a coding. By \ref{lem:all_pairs}\eqref{lem:all_pairs1},  $\tau(\varphi(x))\neq\tau(\varphi(y))$ implies that $\tau(\varphi^m(x))\neq\tau(\varphi^m(y))$ for all $m\geq 1$.
	We cannot, in general, hope to show that $\tau(x)\neq\tau(y)$ implies $\tau(\varphi^{nm}(x))\neq \tau(\varphi^{nm}(y))$ for all $m\geq 0$ for some $n\geq 1$. Indeed, this may fail even when the coding $\tau$ is trivial:  consider any substitution $\varphi$, for which there exist two distinct letters $a$, $b$ such that $\varphi(a)=\varphi(b)$, i.e.\ a substitution $\varphi$, which is not \textit{injective}. If $X_{\varphi}$ is aperiodic then a standard procedure (called \textit{injectivization}) allows one to replace an automatic system $X_{\varphi}$ by a conjugate automatic system $X_{\widetilde{\varphi}}$ given by an injective substitution $\widetilde{\varphi}$ \cite[Prop.\ 2.3]{BDM-04}. Furthermore, if  $X_{\varphi}$ is minimal then  one can find (in a process  called \textit{minimization}) a substitution $\varphi'$ of constant length $k$ such that $\tau(X_{\varphi})$ and $X_{\varphi'}$ are conjugate (see Theorem \ref{thm:MY}).
	 However injectivization does not work well when the system is not aperiodic:  If $X_{\varphi}$ is not aperiodic, then the injectivized systems $X_{\widetilde{\varphi}}$ need not be conjugate to $X_{\varphi}$.
	 If $X_{\varphi}$ is not minimal, then the minimization does not work in general either (see Example \ref{ex:automaticsystemnotsubstitutive}).
	  Some of the additional technicalities in the proofs below are due to the fact that we work with general (not purely) automatic systems without excluding periodic points.

	Let $\varphi$ be a substitution of constant length $k$ and let $\tau$ be a coding (or, a factor map).
	The fact that for every quasi-fixed point $x$ of $\varphi$, $\tau(x)$  is $k$-automatic follows easily from  Proposition \ref{prop:gen_points_are_automatic}\eqref{prop:gen_points_are_automatic2}. Thus, the crucial part in the proof of Theorem \ref{thm:characteraytomaticsequences} is to show that all $k$-automatic points  in $\tau(X_{\varphi})$ are of the required form. 

\begin{lemma}\label{lem:periodicpoints}
Let $\varphi$ be an ambi-idempotent substitution and let $\tau$ be a coding. Let $y\in \tau(X_{\varphi})$ be nonperiodic and let $x\in X_{\varphi}$ be any point such that $\tau(x)=y$. Then all sequences $\tau(x^i)$ are nonperiodic, where $\bm{R_{\varphi}}(x)=(x^i)_i$.
\end{lemma}
\begin{proof}
Let $X^i$, $i\geq 0$ be the orbit closure of $x^i$; note that, by compactness, $\tau(X^i)$ is the orbit closure of $\tau(x^i)$ for all $i\geq 0$.	 
	Suppose that $\tau(x^i)$ (and, thus, $\tau(X^i)$) is periodic for some $i$.  
	Since $x^0$ is equal to $\varphi^{i}(x^i)$ modulo the shift and, by Proposition \ref{thm:subsystems}\eqref{thm:subsystems3}, $\varphi^i(x^i)\in X^i$, we have   $X^0\subset X^i$.
	Hence, $\tau(X^0)\subset \tau(X^i)$ is periodic and, thus, $y$ is periodic as an element of $\tau(X^0)$, which is a contradiction.
\end{proof}

It is necessary to impose some assumptions (such as ambi-idempotency) on a substitution $\varphi$ in  Lemma \ref{lem:periodicpoints};  consider e.g.\ the following example.

\begin{example}\label{ex:periodicpoints} Let  $\varphi\colon \mathcal{A}\rightarrow \mathcal{A}^*$ be the substitution given by $\mathcal{A}=\{0,1,2,3,4\}$ and \[\varphi(0)=0130,\quad \varphi(1)=3443, \quad\varphi(2)=4334,\quad \varphi(3)=1221, \quad\varphi(4)=2112.\] Let $\mathcal{B}=\{0,1,2,3\}$ and let $\tau\colon \mathcal{A}\to \mathcal{B}$ be the coding identifying letters $3$ and $4$, i.e. \[\tau(0)=0,\quad\tau(1)=1,\quad \tau(2)=2,\quad \tau(3)=\tau(4)=3.\] Let $x'={}^{\omega}\!(\varphi^2)(1).(\varphi^2)^{\omega}(2)$ and $x''={}^{\omega}\!(\varphi^2)(3).(\varphi^2)^{\omega}(4)$ be the Thue--Morse sequences over the alphabets $\{1,2\}$ and $\{3,4\}$, respectively. Note that $x'$ and $x''$ lie in $X_{\varphi}$. Then $R_{\varphi}(x')=x''$,  $y=\tau(x')$ is nonperiodic (since $\tau(x')=x'$), and $\tau(x'')$ is periodic (since it is the constant sequence of $3$'s).
\end{example}

\begin{proposition}\label{prop:automaticiffcphiareultimatelyperiodic} Let $k\geq 2$. Let $\varphi$ be a substitution of length $k$ and let $\tau$ be a coding. Assume $\varphi$ is ambi-idempotent and column-constant. Let $y\in \tau(X_{\varphi})$ be nonperiodic and let $x\in X_{\varphi}$ be any point such that $\tau(x)=y$.   The following conditions are equivalent: \begin{enumerate}
\item\label{prop:main1} $y$ is $k$-automatic,
\item\label{prop:main2} $x$ is a quasi-fixed point of $\varphi$.
\end{enumerate}
\end{proposition}
\begin{proof}
	If $x$ is a quasi-fixed point of $\varphi$, it is $k$-automatic by Proposition \ref{prop:gen_points_are_automatic}\eqref{prop:gen_points_are_automatic2}. 
	Hence, $y$ is $k$-automatic as a coding of a $k$-automatic sequence.
	 This shows that \eqref{prop:main2} implies \eqref{prop:main1}. 

	To show the other implication assume that $y$ is $k$-automatic and nonperiodic. 
	Consider any  $x\in X_{\varphi}$  such that $\tau(x)=y$.
	 Let  $\bm{R_{\varphi}}(x)=(x^i)_{i\geq 0}$ and $\bm{c_{\varphi}}(x)=(c_i)_{i\geq 0}$ be the desubstitution sequences of $x$ (over $X_{\varphi}$ and $[0,k-1]$, respectively) such that 
\begin{equation}\label{representation_main}
x^0=x\quad \text{and}\quad x^i=T^{c_{i}}(\varphi(x^{i+1})) \quad \text{for}\quad i\geq 0.
\end{equation} 
By Lemma \ref{lem:periodicpoints}, all $\tau(x^i)$, $i\geq 0$ are nonperiodic. 	 

	 By representation \eqref{representation_main} and Lemma \ref{lem:kernel}, for all $m\geq 1$ and $\bm{i}\in [0,k-1]^m$, the sequence $\Psi_{\bm{i}}(x^m)$ lies either in $K_{k}(x)$ or $T^{-1}(K_k(x))$ (depending on whether $\sum_{l=0}^{m-1} i_{m-l-1}k^l\geq \sum_{l=0}^{m-1} c_lk^l$ or not).  Let $\F=\F_{\varphi,1}$. Since $\varphi$ is column-constant, \[\F_{\varphi,m}=\F \quad \text{for all}\quad m\geq 1;\]
	  in particular,  for any $\Theta\in \F$ and $m\geq 1$, there exists some  $\bm{i}\in [0,k-1]^m$ such that $\Theta=\Psi_{\bm{i}}$. 
	  Hence for any $m\geq 1$ we have 
\[\{\tau(\Theta(x^m))\mid \Theta\in \F\}\subset \mathrm{K}_k(y)\cup T^{-1}(\mathrm{K}_k(y)),\]
where $y=\tau(x)$.

	By Lemma \ref{lem:finitely_many_representations}, there exists $K\geq 1$ such that
	\begin{equation}\label{eq:K}
	\abs{\tau^{-1}(y')}\leq K\quad\text{for any nonperiodic}\quad y'\in Y.
	\end{equation}
	Since $y$ is $k$-automatic, its $k$-kernel $\mathrm{K}_{k}(y)$ is finite and thus the set $ \mathrm{K}_k(y)\cup T^{-1}(\mathrm{K}_k(y))$ is finite. The set $\F$ is finite as well (e.g.\ by Lemma \ref{lem:setFisfinite}). 
	 By pigeonhole principle we can thus find a set $\R\subset \N$ of $K+1$ distinct strictly positive integers such that \begin{enumerate}
\item\label{prop1} for each $\Theta\in \F$, the sequences $\tau(\Theta(x^{r}))$ coincide for all $r\in\R$;
\item \label{prop3} the integers $c_{r-1}$ are the same for all $r\in\R$. 
\end{enumerate}  Since $\F=\F_{\varphi,1}$, property \eqref{prop1} implies that the sequences
\[\tau(\varphi(x^{r}))\] are the same for all $r\in\R$. 
	Since \[\tau(x^{r-1})= T^{c_{r-1}}(\tau(\varphi(x^{r}))),\quad r\in\R\]
property \eqref{prop3} implies that the value $\tau(x^{r-1})$ is independent of $r\in\R$. 
	Call this value $y'$, this is a sequence in $Y$. Since $y'$ is nonperiodic and $\R$ has cardinality $K+1$,  by \eqref{eq:K}, we can find  two integers $p<q$ among  $\{r-1\mid  r\in\R\}$ such that the sequences $x^p$ and $x^q$ are equal.
	Since $x^p$ is equal, modulo the shift, to $\varphi^{q-p}(x^q) = \varphi^{q-p}(x^q)$, $x^p$ is a quasi-fixed point of $\varphi$. 
	Thus $x$ is a quasi-fixed point of $\varphi$, since quasi-fixed points are preserved under $\varphi$ by Proposition \ref{prop:gen_points_are_automatic}\eqref{prop:gen_points_are_automatic1}.
\end{proof}

\begin{proof}[Proof of Theorem \ref{thm:characteraytomaticsequences}]
	We will first show the claim under the additional assumptions that $\varphi$ is ambi-idempotent and column-constant and that the factor map $\pi$ is a coding; we write $\pi=\tau$ in this case. 

Let $y\in Y$. If $y$ is nonperiodic, then all the claims follow directly from Proposition \ref{prop:automaticiffcphiareultimatelyperiodic}.
	 Assume thus that $y$ is periodic. Note that in this case $y$ is $k$-automatic, and we only need to show that there exists some quasi-fixed point $x\in X_{\varphi}$ such that $\tau(x)=y$. Consider the periodic subsystem $Y'\subset Y$ consisting of the orbit of $y$; note that $Y'$ is minimal.
	  By Proposition \ref{thm:subsystems}\eqref{thm:subsystems2}, $Y'=\tau(X_{\varphi'})$ for some primitive substitution 
\[ \varphi' = \varphi|_{\A_b}\colon \A_b\to\A_b^*\]
defined on a subalphabet $\A_b\subset\A$. Since $\varphi$ is ambi-idempotent, $\varphi'$ is also ambi-idempotent, and there exists some fixed point $x'$ of $\varphi'$ such that $X'$ is the orbit closure of $x'$. Since $Y'$ is finite, $y=\tau(x)$ for some $x\in X_{\varphi'}$ that is a shift of $x'$. Since $x'$ is a quasi-fixed point and quasi-fixed points are preserved under shifts this proves the claim.

	We will now deduce the claim when $\pi\colon X_{\varphi}\to Y$ is a general factor map. By the Curtis--Hedlund--Lyndon Theorem (Theorem \ref{thm:CurtisHedlundLyndon}) and Proposition \ref{prop:higherblocksubstitutionproperties}\eqref{prop:higherblocksubstitutionproperties4}, there exists  an odd integer $r=2n+1$, $n\geq 0$ and a coding 
\[\pi_r\colon \widehat{\mathcal{A}_{r}}\to\mathcal{B}\] such that the following diagram
\begin{equation}\label{eq:diagram}
  \begin{tikzcd}
    X_{\varphi} \arrow{r}{\iota_r} \arrow{d}{\pi} & X_{\hat{\varphi_r}} \arrow{d}{\pi_r}   \\
   Y\arrow[leftarrow]{r}{T^{-n}} &  Y
  \end{tikzcd}
\end{equation}
commutes, where $\widehat{\mathcal{A}_{r}}=\mathcal{L}_r(X_{\varphi})$, $X_{\hat{\varphi_r}}$ is the system generated by the $r$th block substitution induced by $\varphi$, and $\iota_r$ is the $r$th higher block presentation map (see Section \ref{sec:factors}).
	 
	 By Proposition \ref{prop:gen_points_are_automatic}\eqref{prop:gen_points_are_automatic0}, $x\in X_{\varphi}$ is a quasi-fixed point of $\varphi$ if and only if $\iota_r(x)$ is a quasi-fixed point of $\hat{\varphi_r}$. 
	 Applying the above case with a trivial coding gives that	$x\in X_{\varphi}$ (resp., $x\in X_{\hat{\varphi_r}}$) is $k$-automatic if and only if it is a quasi-fixed point of $\varphi$ (resp., $\hat{\varphi_r}$).
	Furthermore, by the above case and the fact that shifts preserve $k$-automaticity, the claim holds for the factor map 
	\[T^{-n}\circ\pi_r\colon X_{\hat{\varphi_r}} \to Y\] given by the composition of the coding $\pi_r$ with the shift $T^{-n}$.
	 			This together with the commutativity of the diagram \eqref{eq:diagram} implies all the claims for the factor map $\pi\colon X_{\varphi}\to Y$.
	 			
	 	Now we will deduce the claim for general $\varphi$. 
	 			By Lemma \ref{lem:all_pairs} and Lemma \ref{lem:idempotent}, there is $n$  such that $\varphi^n$ is  column-constant and  ambi-idempotent.
	 			 Furthermore, by Proposition \ref{prop:gen_points_are_automatic}\eqref{prop:gen_points_same_for_power}, the sets of quasi-fixed points of $\varphi$ and quasi-fixed points of $\varphi^n$ coincide for all $n\geq 1$.
	 	This is enough to deduce the claim under the assumption that $X_{\varphi}=X_{\varphi^n}$; this  holds  under very mild assumptions (see Remark \ref{rem:whenpowersystemsagree}).
	 In the general case, one can argue as follows.
	 
	  Let $A'\subset A$ be the set of letters which appear in $\varphi(a)$ for some $a\in\A$ and let 
	  \[\varphi'=\varphi|_{A'}\colon \A'\to (\A')^*
	  \]
	  be the restriction of $\varphi$ to $\A'$. By \cite[Lem.\ 5.3]{BPR23}, $X_{\varphi'}=X_{(\varphi')^{n}}$ for any $n\geq 1$ and the claim holds for $X_{\varphi'}$.
	  We claim that all $x\in X_{\varphi}\setminus X_{\varphi'}$ are shifts of $\varphi$-periodic point and, thus, are automatic.
	  	This will clearly finish the proof.
	  	Repeating the reasoning in the proof of \cite[Thm.\ 2.13]{BKK}  one shows that either $x$ is a shift of ${}^{\omega}\!(\varphi^n)(b).(\varphi^n)^{\omega}(a)$ for $a$ right prolongable  and $b$ left prolongable w.r.t.\ $\varphi^n$ (where $n$ is such that $\varphi^n$ is ambi-idempotent) or $x\in X_{\varphi,c}$ for some $c\in \mathcal{L}(X_{\varphi})$.
	  	Note that the second case implies that $c\in \A'$ and, thus, $x\in X_{\varphi'}$. Hence, the fist case has to hold and $x$ is a shift of a $\varphi$-periodic point.
\end{proof}

\section{Appendix: One-sided substitutive systems}\label{appendix} 

In this appendix we study quasi-fixed points of substitutions in the one-sided context and  obtain a characterisation of automatic sequences in one-sided automatic systems. While the main results and the general structure of the proof is the same as in the two-sided setting, the one-sided case presents some technical difficulties. This is mainly due to the fact that recognisability may fail in the one-sided setting and  (transitive) subsystems of one-sided substitutive systems have more complicated structure; in particular, a subsystem of a one-sided purely automatic system $X_{\varphi}$ may not be closed under $\varphi^n$ for any $n\geq 1$. 

Here we deal exclusively with one-sided systems, and contrary to the rest of the paper all notation refers to the one-sided case.  This change should not lead to any confusion.  We let $\mathcal{O}^{+}(x)=\{T^n(x)\mid n\in \N\}$ denote the forward orbit of a point $x$ in a system $(X,T)$. 
	   A system $X$ is  \textit{transitive} if there exists $x\in X$ such that $\overline{\mathcal{O}^{+}(x)}=X$. 
	    For a substitution $\varphi\colon\mathcal{A}\to\mathcal{A}^*$, we let 
\[X_{\varphi}=\{x\in\mathcal{A}^{\N}\mid \textrm{ every factor of } x \textrm{ appears in } \varphi^n(a) \textrm{ for some } a\in\mathcal{A}, n\geq 0\}\]
denote the one-sided system generated by $\varphi$ and we use the notation
 \[X_{\varphi,b}=\{x\in\mathcal{A}^{\mathbf{N}}\mid \textrm{ every factor of } x \textrm{ appears in } \varphi^n(b) \textrm{ for some }  n\geq 0\},\]
 where $b\in\A$.
	Remark \ref{rem:languageofsubsystems} and \ref{rem:whenpowersystemsagree} hold without any changes with the one-sided definition of $X_{\varphi}$; in particular $X_{\varphi}=X_{\varphi^{n}}$ whenever $X_{\varphi}$ is transitive (by \cite[Lem.\ 1.5]{BKK}) or whenever all letter from $\A$ appear in $\varphi(\A)$ (by the fact that the proof in \cite[lem.\ 5.3]{BPR23} works also for one-sided systems).

	 By a slight abuse of notation, we write $T^c(x)=y$ for some $c<0$ and one-sided sequences $x$ and $y$, if $x=T^{-c}(y)$. The main definition and result are the same as in the two-sided setting. 
\begin{definition}\label{def:generalised_fixed_point_one} Let $\varphi\colon \mathcal{A}\rightarrow \mathcal{A}^*$ be a substitution. A sequence $z\in \mathcal{A}^{\mathbf{N}}$ is called a quasi-fixed point  of $\varphi$ if there exist $m>0$ and $c\in \Z$ such that
\[T^c(\varphi^m(z))=z.\] 
\end{definition}

\begin{theorem} \label{thm:onecharacteraytomaticsequences} Let $k\geq 2$. Let $\varphi$ be a substitution of length $k$. Let $\pi\colon X_{\varphi}\to Y$ be a factor map onto some subshift $Y$. The following hold.\begin{enumerate}
\item \label{thm:onecharacteraytomaticsequences1} A sequence $y\in Y$ is $k$-automatic if and only if $y=\pi(x)$ for some quasi-fixed point $x$ of $\varphi$.
\item \label{thm:onecharacteraytomaticsequences2} If $y\in Y$ is $k$-automatic and nonperiodic, then all  points in $\pi^{-1}(y)\subset X$ are quasi-fixed points of $\varphi$.
\end{enumerate}
\end{theorem}

Theorem \ref{thm:subsystems_one} below gathers the results about  subsystems of one-sided substitutive systems that we will need. To state them, we need the notion of an \emph{idempotent} substitution that appeared in \cite[Def.\ 1.6]{BKK}. We will not use this notion \emph{per se} here; the important thing for us is that for any substitution $\varphi$ some of its power $\varphi^n$, $n\geq 1$ is idempotent \cite[Lem.\ 1.8]{BKK}. For this reason we will not state the (quite lengthy) definition here.

\begin{theorem}\label{thm:subsystems_one}
Let $\varphi\colon\mathcal{A}\rightarrow\mathcal{A}^*$ be an idempotent substitution, and let $\tau\colon\mathcal{A}\to\mathcal{B}$ be a coding. 
\begin{enumerate}
\item\label{thm:subsystems_one3}\cite[Cor.\ 2.7 and Prop.\ 2.6]{BKK}  Let $y\in X_{\varphi}$ and let $Y$ be the orbit closure of $y$. Then one of the following holds: \begin{enumerate}[(a)]
\item $Y=X_{\varphi,b}$ for some $b\in\mathcal{A}$, or
\item $y$ is a quasi-fixed point of $\varphi$.
\end{enumerate} 
\item\label{thm:subsystems_one2}\cite[Prop.\ 2.2 and Lem.\ 1.1]{BKK}   All minimal subsystems of $\tau(X_{\varphi})$ are given by $\tau(X_b)$ for some $b\in\mathcal{A}$, where \[\varphi'=\varphi|_{\mathcal{A}_{b}}\colon\mathcal{A}_{b}\to(\mathcal{A}_{b})^*\] is a primitive substitution.
\end{enumerate}
\end{theorem}

Let $\varphi\colon\mathcal{A}\to\mathcal{A}^*$ be a substitution.  For each $x\in X_{\varphi}$ we denote by $(\bm{c_{\varphi}}(x),\bm{R_{\varphi}}(x))$ the set of all pairs of sequences $(c_i)_{i\geq 0}$ and $(x^i)_{i\geq 0}$ over $\N$ and $X_{\varphi}$, respectively, such that 
\begin{equation}\label{eq:rep_one} x^0=x,\quad  x^i=T^{c_{i}}(\varphi(x^{i+1})),\quad \text{and}\quad 0 \leq c_i<|\varphi(x^{i+1}_0)|.\end{equation}
In particular, $\bm{c_{\varphi}}(x)$ denotes the set of all  sequences $(c_i)_i$, and $\bm{R_{\varphi}}(x)$ the set of all sequences $(x^i)_i$ that can appear in \eqref{eq:rep_one}.  The set  $(\bm{c_{\varphi}}(x),\bm{R_{\varphi}}(x))$ is nonempty for all $x\in X_{\varphi}$ \cite[Lem.\ 1.3]{BKK}. 
	 As mentioned before, recognisability may fail in the one-sided setting and so the set $(\bm{c_{\varphi}}(x),\bm{R_{\varphi}}(x))$ may have more than one element even when $x$ is not ultimately periodic \cite[Rem.\ 2.2(5)]{BSTY}. Nevertheless, Remark \ref{rem:finitecphi} and Lemma \ref{lem:finitely_many_representations} still hold in the one-sided setting.

\begin{lemma}\label{lem:finitely_many_representations_one} Let $\varphi$ be a substitution of constant length and let  $\pi\colon X_{\varphi}\to Y$ be a factor map onto some subshift $Y$. Then, $|\pi^{-1}(y)|\leq K$ for all nonperiodic $y\in Y$ with some constant $K$ independent of $y$.
\end{lemma}
\begin{proof} A proof is exactly the same as in Lemma \ref{lem:finitely_many_representations}, using the one-sided version of Critical Factorisation Lemma: Let $W$ be a finite set of words.
For a one-sided sequence $x$, a $W$-\emph{interpretation} of $x$ is a sequence $(w_i)_{i\geq 0}$ such that $w_i\in W$ for all $w\geq 1$, $w_0$ is a proper suffix of some $w\in W$ and $x$ can be written as 
$x=w_0w_1w_2\dots$.\footnote{We choose the name $W$-interpretation for one-sided-sequences, since $W$-factorisation of a one-sided sequence usually means that $w_0$ is empty; the term $W$-interpretation is standardly  used with finite words.} As in the two-sided case, the set of cuts of a $W$-interpretation $F$ of $x$ is the set of all starting positions of factors $w\in W$ in $x$, i.e.\ it is the set $\{|w_0\dots w_{i}| \mid i\geq 0\}\subset \N$. Two $W$-interpretation of a one-sided sequence are \emph{disjoint} if their sets of cuts are disjoint. The Critical Factorisation Lemma implies \cite[Cor.\ 1]{KM-2002}, \cite[Chapter 8]{Lothaire-book1} that a nonperiodic $x$ can have at most $\abs{W}$ pairwise disjoint interpretations.
\end{proof}

	We say that a one-sided sequence $(x_n)_{n\in\N}$ is a \emph{suffix} of a two-sided sequence $(y_n)_{n\in\Z}$ if there is $m\in\Z$ such that $x_n = y_{m+n}$ for all $n\geq 0$.
	Proposition \ref{prop:char_quasi_periodic_one} below states that all one-sided quasi-fixed points of $\varphi$ arise as suffixes of two-sided quasi-fixed points of $\varphi$, which were characterised in Proposition \ref{prop:char_quasi_periodic_one}.

\begin{proposition}\label{prop:onechar_quasi_periodic}  Let $\varphi\colon \mathcal{A}\rightarrow \mathcal{A}^*$ be a substitution and let $x\in\mathcal{A}^{\N}$. Then $x$  is a quasi-fixed point of $\varphi$ if and only if $x$ is a suffix of a two-sided point $y$ that is of one of the following forms:\begin{enumerate}
\item \label{prop:onechar_quasi_periodic1}  $y= {}^{\omega}\!(\varphi^m)(b).(\varphi^m)^{\omega}(a)$ for some $m\geq 1$, and  $a,b\in\mathcal{A}$ which are, respectively, right and left-prolongable with respect to $\varphi^m$.
\item \label{prop:onechar_quasi_periodic2}  $y=(\varphi^m)^{\omega,i}(a)$ for some $m\geq 1$ and $a\in\mathcal{A}$.
\end{enumerate} 
For a point $x$ to lie in $X_{\varphi}$, we need to further assume in \eqref{prop:char_quasi_periodic_one1} that either $ba\in \mathcal{L}(X_{\varphi})$, or $x$ is a suffix of a one-sided sequence $(\varphi^{m})^{\omega}(a)$. 
\end{proposition}
\begin{proof}The fact that each $x\in\mathcal{A}^{\N}$ which is of the form \eqref{prop:char_quasi_periodic_one1} or  \eqref{prop:char_quasi_periodic_one2} is a quasi-fixed point of $\varphi$ is clear since $x$ is a suffix of a two-sided quasi-fixed point of $\varphi$. To show the other implication, assume that $x$ is a quasi-fixed point of $\varphi$, and write \begin{equation}\label{eq:qp_one}T^c(\varphi^m(x))=x\end{equation} for some $c\in\Z$ and $m\geq 1$. If $c=0$, then $x$ is a fixed point of $\varphi^m$ and the claim holds. Assume that $c\neq 0$. We will show that there exists a prolongation of $x$ to a two-sided sequence $\tilde{x}$ such that $\tilde{x}$ is a two-sided quasi-fixed point of $\varphi$. We will show this in the case when $c>0$; the case $c<0$ is similar.

Since $\varphi$ is growing, there exists a prefix $v$ of $x$ such that $|\varphi^m(v)|>c+|v|$. Equation \eqref{eq:qp_one} implies that $\varphi^m(v)=wvw'$ for some nonempty words $w,w'$ with $|w|=c$ and $w'$ nonempty. Iterating \eqref{eq:qp_one}, we obtain that $x=v\varphi^m(w')\varphi^{2m}(w')\ldots$, and that the two-sided sequence
\[\tilde{x}=\ldots \varphi^{2m}(w)\varphi^{m}(w).v\varphi^m(w')\varphi^{2m}(w')\ldots\] is a prolongation of $x$ such that $T^c(\varphi^m(\tilde{x}))=\tilde{x}$; in particular, $\tilde{x}$ is a two-sided quasi-fixed point of $\varphi$. Thus the claim follows from Proposition \ref{prop:char_quasi_periodic_one}.
\end{proof}

The following proposition can be shown in the same way as (the corresponding claims) in Proposition \ref{prop:gen_points_are_automatic} or deduced from Proposition \ref{prop:gen_points_are_automatic}, Proposition \ref{prop:onechar_quasi_periodic}, and Lemma \ref{lem:closure}\eqref{lem:closure1}.

\begin{proposition}\label{prop:gen_points_are_automatic_one}  Let $\varphi\colon \mathcal{A}\rightarrow\mathcal{A}^*$ be a substitution. \begin{enumerate}
\item\label{prop:gen_points_are_automatic_one1}   The set of quasi-fixed points of $\varphi$  is closed under the right and left shift and under $\varphi$.  
\item\label{prop:gen_points_are_automatic_one2}   Every quasi-fixed point of $\varphi$ is substitutive; furthermore, if $\varphi$ is of constant length $k$, then every quasi-fixed point of $\varphi$ is $k$-automatic.
\end{enumerate}
\end{proposition}

Proposition \ref{prop:main_one} below is the analogue of Proposition \ref{prop:automaticiffcphiareultimatelyperiodic}; the main argument is very similar and we present it in a briefer manner, stressing the points when the proof differs from the one in the two-sided case. For a substitution $\varphi\colon\mathcal{A}\to\mathcal{A}^*$ of  length $k$ and $\bm{i}\in [0,k-1]^*$, we use freely the notation $\Psi_{\bm{i}}\colon \mathcal{A}\to\mathcal{A}$ that was introduced in Definition \ref{def:kernelmaps}. Here, we extend the map $\Psi_{\bm{i}}$ to the map $\Psi_{\bm{i}}\colon \mathcal{A}^{\N}\to\mathcal{A}^{\N}$, and put
\[ \F_{\varphi,m}=\{\Psi_{\mathbf{i}}\colon \mathcal{A}^{\N}\to\mathcal{A}^{\N}\mid \mathbf{i}\in [0,k-1]^m\}.\] Similarly, $\mathrm{K}_k(x)=\{ (x_{i+k^{m}n})_n \mid m\geq 0, 0\leq i\leq k^m-1\}$ stands here for the $k$-kernel of a one-sided sequence $x=(x_i)_{i\in \N}$.  A one-sided sequence $x$ is $k$-automatic if and only if its $k$-kernel is finite. Note that Lemma \ref{lem:kernel} stays true if we replace two-sided sequences by one-sided sequences and two-sided kernels by  one-sided kernels.

Let $\varphi$ be a substitution and let $\tau$ be a coding. For a sequence $x$ in a one-sided substitutive system $X_{\varphi}$, it is no longer true that nonperiodicity of $\tau(x)$ implies  nonperiodicity of all $\tau(x^{mi})$, $i\geq 0$ for some $m\geq 1$: it can happen that all $\tau(x^i)$ except $\tau(x)$ itself are periodic, see  Example \ref{ex:periodicpoints_one} below and compare with Lemma \ref{lem:periodicpoints}.\footnote{It can be shown that this holds if we replace periodic points by ultimately periodic points, but we will not need this.} However, Theorem \ref{thm:subsystems_one}\eqref{thm:subsystems_one3} implies that this can only happen when $x$ is already a quasi-fixed point of $\varphi$. We state this observation in Lemma \ref{lem:periodicpoints_one}.

\begin{example}\label{ex:periodicpoints_one} Let $\mathcal{A}=\{0,1,2,3\}$ and let $\mathcal{A}\to\mathcal{A}^*$ be the substitution given by 
\[\varphi(0)=1023,\quad \varphi(1)=1201, \quad \varphi(2)=2332,\quad \varphi(3)=3223.\]
Let $\mathcal{B}=\{0,1\}$ and let $\tau\colon\mathcal{A}\to\mathcal{B}$ be the coding identifying letters $0$, $2$ and $3$, i.e.\
\[\tau(1)=1,\quad \tau(0)=\tau(2)=\tau(3)=0.\] Let $w=1$, $v=23$ and write $\varphi(0)=w0v$. Let \[x=10v\varphi(v)\varphi^2(v)\dots,\quad \text{and}\quad x'=T(x)=0v\varphi(v)\varphi^2(v)\dots.\] Both $x$ and $x'$ lie in $X_{\varphi}$. Note that $x=\varphi(x')$ and $x'=T(\varphi(x'))$,  in particular, the sequence $(x^i)_i=(x,x',x',\dots)$ lies in $\bm{R_{\varphi}}(x)$. However, $\tau(x)=10^{\omega}$ is not periodic, while $\tau(x')=0^{\omega}$ is periodic. 
\end{example}

\begin{lemma}\label{lem:periodicpoints_one} Let $\varphi\colon \mathcal{A}\rightarrow \mathcal{A}^*$ be an  idempotent substitution, let $\tau\colon\mathcal{A}\rightarrow\mathcal{B}$ be a coding, let $Y=\tau(X_{\varphi})$, and let $y\in Y$. Assume that $y$ is not periodic and consider any $x\in X_{\varphi}$ such that $\tau(x)=y$. Let $((c_i)_i,(x^i)_i))$ be any pair in $(\bm{c_{\varphi}}(x),\bm{R_{\varphi}}(x))$. Then one of the following holds:
\begin{enumerate}
\item\label{lem:periodicpoints_one1}$\tau(x^i)$ is nonperiodic for all $i\geq 0$, or
\item\label{lem:periodicpoints_one2} $x$ is a quasi-fixed point of $\varphi$.
\end{enumerate}
\end{lemma}
\begin{proof} 
	For each $i\geq 0$, let $X^i$ denote the orbit closure of $x^i$. We consider two cases.\newline
\textbf{Case 1} (For each $i\geq 0$ there exists $b\in\mathcal{A}$ such that $X^i=X_{\varphi,b}$.) Since each $X_{\varphi,b}$ is closed under $\varphi$, we have that, in fact, $X^i=X_{\varphi,b}$ for the same letter $b$ for all $i\geq 0$. Now repeating the proof in Lemma \ref{lem:periodicpoints}, we get that if $\tau(x^i)$ is periodic for some $i$, then $y=\tau(x)$ is periodic, which is a contradiction. Hence, case \eqref{lem:periodicpoints_one1} holds.\newline
\textbf{Case 2} (There exists $i\geq 0$ such that $X^i$ is different from $X_{\varphi,b}$ for any $b$.) In this case, by Theorem \ref{thm:subsystems_one}\eqref{thm:subsystems_one3}, $x^i$ is a quasi-fixed point of $\varphi$. Since the set of quasi-fixed points of $\varphi$ is closed under $\varphi$ and the shift, $x$ is a quasi-fixed point of $\varphi$. Hence, case \eqref{lem:periodicpoints_one2} holds.
\end{proof}

\begin{proposition}\label{prop:main_one}  Let $k\geq 2$. Let $\varphi$ be a substitution of length $k$ and let $\tau$ be a coding. Assume $\varphi$ is idempotent and column-constant. Let $y\in \tau(X_{\varphi})$ be nonperiodic and let $x\in X_{\varphi}$ be any point such that $\tau(x)=y$.   The following conditions are equivalent:  \begin{enumerate}
\item\label{prop:main_one1}  $y$ is automatic,
\item\label{prop:main_one2}  $x$ is a quasi-fixed point of $\varphi$.
\end{enumerate}
\end{proposition}
\begin{proof}
If $x$ is a quasi-fixed point of $\varphi$, then, by Proposition \ref{prop:gen_points_are_automatic_one}, it is $k$-automatic. Then $y=\tau(x)$ is also $k$-automatic, which shows that \eqref{prop:main_one2} implies \eqref{prop:main_one1}. 

	To show the other implication assume that $y=\tau(x)$ is $k$-automatic and nonperiodic.
	Let $((c_i)_i,(x^i)_i))$ be any pair in $(\bm{c_{\varphi}}(x),\bm{R_{\varphi}}(x))$; we have
\begin{equation}\label{representation}
x^0=x\quad  \text{and}\quad x^i=T^{c_{i}}(\varphi(x^{i+1})) \quad \text{for}\quad i\geq 0.
\end{equation} 
	By Lemma  \ref{lem:periodicpoints_one}, either $x$ is a quasi-fixed point of $\varphi$ (and hence, is $k$-automatic), or all $\tau(x^i)$ are nonperiodic. In the first case, the claim holds; thus we may assume without loss of generality that the second case holds. 
	Let $\F=\F_{\varphi,1}$. As in the proof of Proposition \ref{prop:automaticiffcphiareultimatelyperiodic}, we show  that for any $m\geq 1$ and $\Theta\in \F$, either $\tau(\Theta(x^m))$, or $\tau(\Theta(T(x^m)))$ lies in $\mathrm{K}_k(y)$. The rest of the proof proceeds the same as in Proposition \ref{prop:automaticiffcphiareultimatelyperiodic} with only small caveats. 
	
	Let $K$ be the constant from Lemma \ref{lem:finitely_many_representations_one}. 	
	By pigeonhole principle, we can find a set $\mathrm{R}\subset \N$ of $K+1$ positive distinct integers  such that:\begin{enumerate}
\item\label{prop_one1}  for each $\Theta\in \mathrm{F}$, the sequences $\tau(\Theta(T(x^{r})))$ coincide for all $r\in\mathrm{R}$;
\item \label{prop_one3} the integers $c_{r-1}$ are the same for all $r\in\mathrm{R}$;
\item\label{prop_one4}  $x^{r}$ have the same initial letter for all $r\in\mathrm{R}$.
\end{enumerate}  
	Since $\mathrm{F}=\mathrm{F}_{\varphi,1}$, property \eqref{prop_one1} implies that the sequences
\[\tau(\varphi(T(x^{r})))\] are the same for all $r\in\mathrm{R}$. 
	Since \[\tau(x^{r-1})= T^{c_{r-1}}(\tau(\varphi(x^{r}))) = T^{c_{r-1}}(\tau(\varphi(x_0^{r})))\tau(\varphi(T(x^{r}))),\]
properties \eqref{prop_one3} and \eqref{prop_one4} imply that the value $\tau(x^{r-1})$ is independent of $r\in\mathrm{R}$. 
	Call this value $y'$, this is a sequence in $X$. By Lemma \ref{lem:finitely_many_representations_one}, we can find  two integers $p<q$ among  $\{r-1\mid  r\in\R\}$ such that the sequences $x^p$ and $x^q$ are equal.
	 It follows that $x^p$ is a quasi-fixed point of $\varphi$, and hence, $x$ is a quasi-fixed point of $\varphi$, since quasi-fixed points are preserved under $\varphi$ and shifts.
\end{proof}

We note here that Proposition \ref{prop:higherblocksubstitutionproperties} holds if we work with the one-sided shifts $X_{\varphi}$ and the one-sided $r$th higher presentation map $\iota_r\colon \A^{\N}\to\widehat{\A_{r}}^{\N}$.

\begin{proof}[Proof of Theorem \ref{thm:onecharacteraytomaticsequences}] 
	Under the assumption that the factor map $\pi$ is a coding the proof is analogous to the one in Theorem \ref{thm:characteraytomaticsequences} using Proposition \ref{prop:main_one}, Theorem \ref{thm:subsystems_one}\eqref{thm:subsystems_one2}, \cite[Lem.\ 5.3]{BPR23} for one-sided systems, and the fact that  $\varphi$ has some power which is idempotent and column-constant.
		For a general factor map $\pi\colon X_{\varphi}\to Y$, by the one-sided version of Curtis--Hedlund--Lyndon Theorem and Proposition \ref{prop:higherblocksubstitutionproperties}\eqref{prop:higherblocksubstitutionproperties4}, one gets that there exist an integer $r\geq 1$  and a coding 
\[\pi_r\colon \widehat{\mathcal{A}_{r}}\to\mathcal{B}\] such that the following diagram
\[
  \begin{tikzcd}
    X \arrow{r}{\iota_r} \arrow{d}{\pi} & X_{\hat{\varphi_{r}}} \arrow{dl}{\pi_r}   \\
      Y
  \end{tikzcd}
\]
commutes, where $\widehat{\mathcal{A}_{r}}=\mathcal{L}_r(X)$ and $\iota_r$ is the one-sided $r$th higher block presentation map. 
	(Note that here $r$ can be any (not necessarily odd) integer and one does not need the additional shift $T^n$). The rest of the proof is the same.
\end{proof}

\bibliographystyle{amsplain}
\bibliography{bibliography}
\end{document}